\documentclass[12pt,reqno]{amsart}
\usepackage[foot]{amsaddr}
\usepackage{amsfonts,amssymb,amsmath}
\usepackage
[
        left=2.5cm,
        right=2.5cm,
        top=3cm,
        bottom=3cm,
]
{geometry}

\newtheorem{definition}{Definition}[section]
\newtheorem{theorem}{Theorem}[section]
\newtheorem{lemma}{Lemma}[section]
\newtheorem{corollary}{Corollary}[section]

\newtheorem{remark}{Remark}[section]

\begin{document}

\sloppy

\title[Stochastic evolution equations with Wick-analytic nonlinearities]
{Stochastic evolution equations with Wick-analytic nonlinearities}

\author{Tijana LEVAJKOVI\'C \and Stevan PILIPOVI\'C \and Dora SELE\v SI \and Milica \v ZIGI\'C}

\address[Tijana Levajkovi\'c]{Applied Statistics Research Unit, Institute of Statistics and Mathematical Methods in Economics,  TU Vienna, Austria.}
\address[Stevan Pilipovi\'c, Dora Sele\v si, Milica \v Zigi\'c]{Department of Mathematics and Informatics, Faculty of Sciences, University of Novi Sad, Serbia.}

\email{tijana.levajkovic@tuwien.ac.at}
\email{stevan.pilipovic@dmi.uns.ac.rs}
\email{dora.selesi@dmi.uns.ac.rs}
\email[Corresponding author]{milica.zigic@dmi.uns.ac.rs}

\thanks{Corresponding author: Milica \v Zigi\'c, milica.zigic@dmi.uns.ac.rs}

\date{05.02.2021.}

\begin{abstract}
We study nonlinear stochastic partial differential equations with Wick-analytic type nonlinearities set in the framework of white noise analysis.  These equations include the stochastic  Fisher--KPP equations,   stochastic Allen--Cahn, stochastic Newell--Whitehead--Segel, and stochastic Fujita--Gelfand equations.  By implementing the theory of $C_0-$semigroups and evolution systems into the chaos expansion theory in infinite dimensional spaces, we  prove existence and uniqueness of solutions for this class of stochastic partial differential equations. 
\end{abstract}



\maketitle

\section{Introduction}\label{0.0}

As a conclusion to our previous papers \cite{Milica} and \cite{Milica2} we perform a follow-up study to extend the results obtained therein. In \cite{Milica} we treated the case of stochastic evolution equations with Wick-multiplicative noise of the form $u_t = {\mathbf A} u +   {\mathbf B} \lozenge u +f$, where ${\mathbf A}$  was a densely defined unbounded operator generating a $C_0$-semigroup (e.g. a strictly elliptic second order partial differential operator), while ${\mathbf B}$ was a linear bounded operator modeling convolution-type perturbations via the Wick product in the equation. Examples of  these equations included the heat equation with random potential, the Schr\"odinger equation, the transport equation driven by white noise, and many others. Afterwards, in  \cite{Milica2} we studied stochastic evolution equations with Wick-polynomial nonlinearities of the form $u_t = {\mathbf A} u + u^{\lozenge n} +f$, including several reaction--diffusion examples such as the Fisher--KPP equations, FitzHugh--Nagumo equations, and many more. In both papers we proved existence and uniqueness of the solutions and also developed an exact algorithm for obtaining explicit solutions in terms of generators of a corresponding $C_0$-semigroup and polynomial chaos expansion. In the current paper we study fully nonlinear stochastic evolution equations of the form $u_t = {\mathbf A} u +   \Phi^{\lozenge}(u) +f$, where the nonlinearity is modeled by an analytic (entire) function $\Phi$, and $\Phi^{\lozenge}$ denotes its corresponding Wick-version as it will be described in detail below. Examples of these equations involve stochastic versions of the  Allen--Cahn, Newell--Whitehead--Segel equations, and various other equations.

Let $X$ be a certain Banach algebra. Let $\Phi$ be an analytic function on $X$ with its Maclaurin expansion $\displaystyle\Phi(\xi)=\sum_{n\geq 0}a_n \xi^n,\;\xi\in X,$ where $a_n,\;n\geq 0$ are either constants or deterministic functions (i.e. elements of $X$). Here we implicitly assumed that $X$ is a Banach algebra of deterministic functions over some field of constants. We study stochastic nonlinear evolution equations of the form  
\begin{align}\label{PNLJ}
u_t (t, \omega)&= \mathbf A \, u(t, \omega) + \Phi^\lozenge \big(u (t, \omega)\big)+f(t,\omega), \quad t\in (0,T], \quad \omega\in\Omega,\nonumber\\
&= \mathbf A \, u(t, \omega) + \sum_{n=0}^\infty a_n u^{\lozenge n}(t, \omega)+f(t,\omega), \quad t\in (0,T], \quad \omega\in\Omega,\\\nonumber
u(0,\omega) &= u^0(\omega), \quad \omega\in\Omega,
\end{align} 
where $u(t,\omega)$, $t\in [0, T]$, $\omega\in \Omega$ is an $X-$valued generalized stochastic process and $\mathbf A$ corresponds to a densely defined infinitesimal  generator of a $C_0-$semigroup.  The nonlinear part $\Phi^\lozenge \big(u (t, \omega)\big)$ is given in terms of Wick-powers $u^{\lozenge n} = u^{\lozenge n-1} \lozenge u=u\lozenge \dots \lozenge u,\;n\in \mathbb{N}$, where $\lozenge$ denotes the Wick product, instead of a classical product considered in the definition of the function $\Phi$.  The Wick product is involved due to the fact that we allow random terms to be present both in the initial condition $u_0$ and the driving force $f$, hence the solution $u$ will inherit these random properties. This leads to singular solutions that do not allow to use ordinary multiplication, but require a renormalization of the multiplication, which is done by introducing the Wick product into the equation. The Wick product is known to represent the highest order stochastic approximation of the ordinary product due to the formula $u\cdot v = u\lozenge v + \sum_{n=1}^\infty\frac{{\mathbb D}^nu\lozenge {\mathbb D}^nv}{n!}$ \cite{Mikulevicius}, while some better approximations may also  be achieved \cite{Mikulevicius2} in the framework of Malliavin calculus by including higher order Malliavin derivatives ${\mathbb D}^n$, $n\in\mathbb N$, and additional terms of the form $\frac{{\mathbb D}^nu\lozenge {\mathbb D}^nv}{n!}$, $n=1,2,\ldots$  into the approximation. Although the Wick product of random variables as a modelling tool in SPDEs \cite{Wick} has received some critiques \cite{kritika}  for having non-intuitive properties such as factorizing expectations and forcing independence it has been used in many papers not just as a simplification tool, but as a necessity for dealing with products and nonlinear functions of distributions where the ordinary product would fail to be well defined.  Wick powers in stochastic equations and Wick-type renormalization is reviewed in \cite{DaPratoTubaro}. Several other theories allow for a formal treatment of nonlinearities of generalized stochastic processes such as Colombeau algebras of generalized stochastic processes \cite{Albeverio, Ober}, also combinations of the latter with Wick calculus \cite{Dora}, rough-paths theory \cite{rough} together with its extensions,  Hairer's celebrated model of regularity structures \cite{Hairer}. Several papers deal with a comparison of the Wick product and the ordinary product, e.g. in terms of an error estimate \cite{Mikulevicius2}, numerical simulations \cite{Theting}, as a scaling limit process \cite{Grothaus} etc. in those cases where the ordinary product is meaningful. Generally, the Wick product has shown to be a satisfactory approximation of the ordinary product from a modelling point of view for small noise levels \cite{Mikulevicius2} and for noise processes with a low correlation \cite{Wan}.

Various problems appearing  in quantum mechanics,  plasma physics, mathematical biology,  ecology, and medicine can be formulated in the form \eqref{PNLJ}. In particular, in stochastic Fisher--KPP equations,  Fujita-type equations, FitzHugh--Nagumo equations nonlinearities can be  generated by Wick-polynomial nonlinearities. In our previous paper \cite{Milica2} we proved existence and uniqueness of solutions of  nonlinear parabolic stochastic partial differential equations with Wick-power and Wick-polynomial type nonlinearities. 

Stochastic reaction-diffusion equation perturbed by multiplicative noise  with a polynomial nonlinearity of the form 
\begin{equation}
\Phi(s) =\sum_{i=1}^{2p-1} a_i \, s^i, \quad \text{with} \, \, a_{2p-1}< 0,
\label{special Phi}
\end{equation}
in a bounded domain, where  $p$ is a positive integer and $a_1, \dots, a_{2p-1}$ are real numbers  was considered  in \cite{Flandoli}.  Based on semigroups techniques, the author proved global existence and uniqueness of such problems.  Stochastic nonlinear equations of the form \eqref{PNLJ} include also the  problems  $u_t = \triangle u +  \Phi^{\lozenge}(u) + u \lozenge W_t$, with $\Phi^\lozenge$ being the Wick version of the  polynomial nonlinearity \eqref{special Phi}. 

Polynomial nonlinearities of the form \eqref{special Phi} appear in a special reaction-diffusion equation, the stochastic Allen--Cahn equation  (also called the stochastic Ginzburg--Landau equation),  $u_t = \triangle u +  u -   u^{\lozenge 3}$. The deterministic Allen-Cahn equation  is used to describe phase separation and the evolution of interfaces between the phases for systems without mass conservation \cite{Allen-Cahn1,Allen-Cahn2}, while the corresponding stochastic Allen--Cahn equations take into account also the effects due to the thermal fluctuations of the system \cite{Bertacco}. The stochastic Newell-Whitehead equation $u_t = \triangle u + a u -  b u^{\lozenge n}$, $a, b>0$, $n>1$,  is also a reaction-diffusion equation which  belongs to the class \eqref{PNLJ}.  The deterministic Newell--Whitehead--Segel  equation models the interaction of the effect of the diffusion term with the nonlinear effect of the reaction term. It has found various applications  in mechanical and chemical engineering, ecology, biology and bioengineering.  

Exponential type nonlinearities appear in the stochastic Fujita--Gelfand equation in combustion, i.e. the stochastic  exponential reaction-diffusion equation  of the form $u_t = \triangle u + \lambda e^{\lozenge u}$, $\lambda >0$. The corresponding deterministic  problems are used to model convection in a heat transfer process when exothermic chemical reactions occur in the fluid body \cite{Jones}.  Exponential type nonlinearity arises in the context of Liouville quantum gravity in models of the form $u_t=1/(4\pi) \Delta u +e^{\gamma u}+\xi $ with space-time white noise $\xi$ \cite{Garban}.

The nonlinear Schr\"odinger  equation with logarithmic nonlinearity possesses solitonlike solutions of a Gaussian shape, also called Gaussons,  in any number of dimensions. These solutions in nonlinear wave mechanics describe the propagation of wave packets of freely moving particles and were explained for example in \cite{log}.  

Existence and uniqueness of nonnegative solutions  of the heat equation with a logarithmic nonlinearity was proven in \cite{Alfaro}.  In the current paper, as an application of the general class of problems \eqref{PNLJ}, we solve stochastic heat equations with logarithmic nonlinearity of the form $u_t = \triangle u + u \lozenge \log^{\lozenge} |u| + f$. Other types of nonlinearities $\Phi$ which are analytic functions, such as trigonometric or hyperbolic functions also fit to our setting. Namely, the  nonlinear stochastic equation of the form $u_t = \triangle u + \cos^{\lozenge} u + \cosh^{\lozenge} u + f$ belongs to the considered class of stochastic equations \eqref{PNLJ}.

In \cite{Potthoff}, an equation with Wick-analytical nonlinearity was examined in a more generic form: $u_t = Au + \Phi^\lozenge(u)+u\lozenge N$ using an elliptic second order differential operator $A$. The $S-$transform and fixed point methods were used to derive the solution.  Our equation under consideration \eqref{PNLJ} uses the semigroup approach and its formulation is in an abstract Cauchy problem form; while prevalent applications primarily involve heat equation scenarios governed by second-order elliptic partial differential operators, our study and methodology extend to encompass a broader spectrum of abstract operators $\mathbf A$ that are generators of $C_0-$semigroups. This includes operators involving first-order spatial derivatives, convolution-type operators like integral transforms, as well as multiplication and translation operators in conjunction with the aforementioned ones. Notably, our study also accommodates the examination of delay equations within this framework.

We implement the Wiener-It\^o chaos expansion method from white noise analysis combined with $C_0-$ operator semigroup theory in order to prove the existence and the uniqueness of a solution to \eqref{PNLJ}. Using the chaos expansion method the initial stochastic partial differential equation (SPDE) is transformed into a lower triangular infinite system of PDEs (propagator system) that can be solved recursively. Solving this system, one obtains the coefficients of the solution to \eqref{PNLJ}. Finally, convergence of the series needs to be proven, and for this purpose we will define a new class of weighted Hida-type spaces and utilize some estimates for growth rates of a factorial function.

The propagator system method has been successfully implemented in \cite{Milica, Milica2, LPS, Mikulevicius, ps}. Our approach provides a way to solve nonlinear stochastic equations \eqref{PNLJ} explicitly. It is important to note that SPDEs with Wick-multiplication have been considered since the very beginning of the foundation of Hida's white noise analysis \cite{Hida} and continued in the papers of Holden, \O ksendal and their coworkers in \cite{HO1, HO2, HO3}, etc. as summed up in \cite{HOUZ}. Our approach differs in one important fact, namely in \cite{HOUZ} the solution is provided in terms of the Hermite transform and its inverse. Similarly as it is the case with the Laplace transform and other integral transforms, the most difficult step is always taking the inverse transform. The setup is comparable in many ways: convolutions are transformed into ordinary products by the Laplace transform, and the convolution-type Wick product is transformed into an ordinary product  by the Hermite transform, which also converts the SPDE into a deterministic PDE with classical multiplication operators. For the existence of the inverse Hermite transform the corresponding (deterministic) PDE needs to satisfy some very strict analytical properties. Sufficient conditions must always be proposed such that they allow the solution of the PDE to fulfill these analytic conditions and hence to be invertible by the Hermite transform which constitutes the main part of SPDE solving by this approach \cite{HOUZ}. Furthermore, one does not gain an explicit formula for the solution, only a theoretical result stating the existence and uniqueness of the solution expressed by a theoretical inverse Hermite transform. In our approach there is no need to apply the Hermite transform (or its counterpart the $S-$transform), one can just follow a recursive pattern to calculate the coefficients of the solution. This is an algorithmic procedure that provides an explicit/computable form of the solution and therefore it  can be  also used as an efficient method for  numerical implementations. Instead of dealing with sufficient conditions that allow the inverse Hermite transform to be applied, the key component in our approach is dealing with appropriate weight factors in the chaos expansion series and hence adequately tailored topologies in spaces of generalized stochastic processes that can accommodate the singularity level of the solution obtained by the recursive propagator method. Since one is not limited to classical complex analysis and analytic functions but rather has the freedom to enter the series expansion theory of ultradistributions and hyperfunctions and their respective weight convergence rates, this methodology enables a much larger playground and further generalizations to even more singular SPDEs.

In the existing literature, certain types of these equations were treated only numerically. For example, an efficient numerical approximation of nonlinear parabolic SPDEs with additive noise  was proposed in \cite{Jentzen}.  The authors obtained higher convergence order in comparison to convergence results of classical schemes such as the linear implicit Euler scheme. An optimal strong convergence rate of a fully discrete numerical scheme for the stochastic Allen--Cahn equation driven by an additive space-time white noise was established in \cite{LiuQiao}. An accurate and efficient 
Haar transform or Haar wavelet method was introduced in  \cite{Hariharan} for solving deterministic nonlinear parabolic partial differential equations such as for example  Allen--Cahn or  Newell--Whitehead--Segel equations. 

The paper is organized as follows: In the introductory section we recall upon basic notions of $C_0-$semi\-groups, evolution systems, entire functions defined on Banach spaces and white noise theory including chaos expansions of generalized stochastic processes. Here we will introduce a space of weighted Hida generalized random variables. In Section \ref{1.0}, which represents the main part of the paper, we prove existence and uniqueness of the solution to \eqref{PNLJ}. In Section \ref{2.0} we provide some concrete examples of nonlinear problems \eqref{PNLJ}.

\subsection{Generalized stochastic processes}\label{0.2}

Denote by $(\Omega, \mathcal{F}, \mu) $ the Gaussian white noise probability space $(S'(\mathbb{R}), \mathcal{B}, \mu), $ where $\Omega=S'(\mathbb{R})$ denotes the space of tempered distributions,  $\mathcal{B}$ is the Borel sigma-algebra generated by the weak topology on $S'(\mathbb{R})$ and $\mu$ the Gaussian white noise measure corresponding to  the characteristic function
\begin{equation*}\label{BM theorem}
\int_{S'(\mathbb{R})} \,  e^{{i\langle\omega, \phi\rangle}} d\mu(\omega) = \exp \left [-\frac{1}{2} \|\phi\|^2_{L^2(\mathbb{R})}\right], \quad \quad\phi\in  S(\mathbb{R}),
\end{equation*}
given by the Bochner-Minlos theorem.

We recall the notions related to   $L^2(\Omega,\mu)$, the space of square integrable random variables over the underlying probability space $(\Omega, \mathcal F, \mu)$, see \cite{HOUZ}.  The set of multiindices $\mathcal I$ is $(\mathbb N_0^\mathbb N)_c$, i.e. the set of sequences of non-negative integers which have only finitely many non-zero components. Especially, we denote by $\mathbf 0=(0,0,0,\ldots)$ the zero multiindex with all entries equal to zero, the length of a multiindex is $|\alpha|=\sum_{i=1}^\infty\alpha_i$ for $\alpha=(\alpha_1,\alpha_2,\ldots)\in\mathcal I$ and $\alpha!=\prod_{i=1}^\infty\alpha_i!.$ We will use the convention that $\alpha-\beta$ is defined if $\alpha_n-\beta_n\geq 0$ for all $n\in\mathbb N$, i.e., if $\alpha-\beta\geq\mathbf 0.$

The Wiener-It\^o theorem states that one can define an orthogonal basis $\{H_\alpha\}_{\alpha\in\mathcal I}$ of $L^2(\Omega,\mu)$, where $H_\alpha$ are constructed by  means of Hermite orthogonal polynomials $h_n$ and Hermite functions $\xi_n$,
\[
H_\alpha(\omega)=\prod_{n=1} ^\infty h_{\alpha_n}(\langle\omega,\xi_n\rangle),\quad \alpha=(\alpha_1,\alpha_2,\ldots, \alpha_n\ldots)\in\mathcal I,\quad \omega\in\Omega.
\]
Then, every $F\in L^2(\Omega,\mu)$ can be represented via the so called \emph{chaos expansion}
\[
F(\omega)=\sum_{\alpha\in\mathcal I} f_\alpha H_\alpha(\omega), \quad \omega\in S'(\mathbb{R}),\quad\sum_{\alpha\in\mathcal I} |f_\alpha|^2\alpha!<\infty,\quad f_\alpha\in\mathbb{R},\quad\alpha\in\mathcal I.
\]

Denote by $\varepsilon_k=(0,0,\ldots, 1, 0,0,\ldots),\;k\in \mathbb{N}$ the multiindex with the entry 1 at the $k$th place. Denote by $\mathcal H_1$ the subspace of $L^2(\Omega,\mu)$ spanned by the polynomials $H_{\varepsilon_k}(\cdot)$, $k\in\mathbb N$. The subspace $\mathcal H_1$ contains Gaussian stochastic processes, e.g. Brownian motion is given by the chaos expansion $B(t,\omega) = \sum_{k=1}^\infty \int_0^t \xi_k(s)ds\;H_{\varepsilon_k}(\omega).$

Changing the topology on $L^2(\Omega,\mu)$ to a weaker one, T. Hida \cite{Hida} defined spaces of generalized random variables containing the white noise as a weak derivative of the Brownian motion. We refer to \cite{Hida, HOUZ} for white noise analysis.

Let $(2\mathbb N)^{\alpha}=\prod_{n=1}^\infty (2n)^{\alpha_n},\quad \alpha=(\alpha_1,\alpha_2,\ldots, \alpha_n,\ldots)\in\mathcal I.$ We will often use the fact that the series $\sum_{\alpha\in\mathcal I}(2\mathbb N)^{-p\alpha}$ converges  for $p>1$. Using the same technique as in \cite[Chapter 2]{HOUZ} one can define Banach spaces $(FS)_{r,p}$ of test functions and their topological duals $(FS)_{-r,-p}$ of stochastic distributions for all $r\geq 2$ and $p\geq 0.$

\begin{definition} Let $r\geq 2,\;p\geq 0.$ We define  the stochastic test function space  by
\[
(FS)_{r,p} =\{F=\sum_{\alpha\in\mathcal I}f_\alpha {H_\alpha}\in L^2(\Omega,\mu):\;  \|F\|^2_{(FS)_{r,p}}= \sum_{\alpha\in\mathcal I} \alpha!|f_\alpha|^2r^{|\alpha|^3}!(2\mathbb N)^{p\alpha}<\infty\},
\]
and its topological dual, the stochastic distribution space,  by formal sums
\[
(FS)_{-r,-p} =\{F=\sum_{\alpha\in\mathcal I}f_\alpha {H_\alpha}:\;  \|F\|^2_{(FS)_{-r,-p}}= \sum_{\alpha\in\mathcal I}\alpha!|f_\alpha|^2(r^{|\alpha|^3}!)^{-1}(2\mathbb N)^{-p\alpha}<\infty\}.
\]
The space of test random variables is
\[
(FS) =\bigcap_{r\geq 2,\;p\geq 0}(FS)_{r,p}
\] 
endowed with the projective topology.  Its dual, the space of generalized random variables is 
\[
(FS)' =\bigcup_{r\geq 2,\;p\geq 0}(FS)_{-r,-p}
\] 
endowed with the inductive topology. 
\end{definition}

The action of  $F=\sum_{\alpha\in\mathcal{I}}  b_\alpha H_\alpha\in (FS)'$ onto $f=\sum_{\alpha\in\mathcal{I}}c_\alpha H_\alpha\in(FS)$ is given by $\langle F,f\rangle=\sum_{\alpha\in\mathcal{I}} \alpha! b_\alpha c_\alpha,$ where $b_\alpha, c_\alpha\in \mathbb R$. Thus, they form a Gel'fand triplet
\[
(FS) \subset L^2(\Omega,\mu) \subset (FS)'.
\]
Particularly, it also holds  $(FS)_{-r,-p} \subseteq (FS)_{-r,-q} $ for $0\leq p\leq q$, $r\geq 2$ and $(FS)_{-r_1,-p} \subseteq (FS)_{-r_2,-p}$ for $r_2\geq r_1\geq 2$, $p\geq0$. 

We will refer to $(FS)$ and $(FS)'$ as to factorial-weighted Hida spaces.

\begin{remark}Note that for $r=1$ one retrieves the classical Hida spaces \cite{Hida},  which are usually denoted as $(S)$ and $(S)^*$.

Other spaces that often occur in the literature are the Kondratiev test spaces $(S)_\rho$ that use the weight factors $\alpha!^{1+\rho} (2\mathbb N)^{p\alpha}$ and Kondratiev distribution spaces $(S)_{-\rho}$ that use the weight factors $\alpha!^{1-\rho} (2\mathbb N)^{-p\alpha}$, for $\rho\in[0,1]$ and construct projective and inductive limit topologies respectively \cite{HOUZ,Potthoff}, or the exponentially weighted spaces with weight factors $\alpha!^{1\pm\rho} Exp_k(\pm p(2\mathbb N)^{\alpha})$, for $\rho\in[0,1]$,  $k\in\mathbb N$, and $Exp_k(x)=\underbrace{\exp(\exp(\cdots (\exp x)\cdots))}_k$, denoted by $Exp_k(S)_{\pm\rho}$ as in \cite{GRPW,Habib}, as well as many other possible weight factors \cite{Sergey}.

Clearly,
$$Exp_k(S)_{\rho }\subseteq (FS)\subseteq (S)_{\rho}\subseteq (S)\subseteq L^2(\Omega,\mu)\subseteq (S)^* \subseteq(S)_{-\rho}\subseteq (FS)'\subseteq Exp_k(S)_{-\rho}.$$

The Wick-analytical nonlinearity in equation \eqref{special Phi} cannot be handled by the classical Hida-Kondratiev spaces of generalized random variables; on the other hand, the exponentially weighted spaces might prove to be excessively large. The topology of the $(FS)'$ space can control the ideal rate of convergence to manage the singularity level of the solution to \eqref{special Phi}.
\end{remark}

\subsubsection{Nuclearity of the space of test random variables $(FS)$ and its dual $(FS)'$}
\label{0.2.1}

The spaces $(FS)_{r,p}$ and $(FS)_{-r,-p},\;r\geq 2,\;p\geq 0$ are separable Hilbert spaces.  Moreover, $(FS)$ and $(FS)'$ are nuclear spaces as one can infer form the following lemma.

\begin{lemma} 
The space of the stochastic test functions $(FS)$ is nuclear. 
\end{lemma}

\begin{proof}
The space $(FS)_{r, p}$ is a Hilbert space with the inner product
\[
(f, g)_{r, p} = \sum_{\alpha\in \mathcal I} f_\alpha\, g_\alpha\,\alpha! \, \,  r^{|\alpha|^3}! \,\, \,  \, (2\mathbb N)^{p\alpha}
\]
for all $f=\sum_{\alpha\in \mathcal I} f_\alpha\, H_\alpha \in (FS)_{r, p}$ and $g=\sum_{\beta\in \mathcal I} g_\beta\, H_\beta \in (FS)_{r, p}$, $r\geq 2$, $p\geq 0$. Thus, the family of functions 
\[
K_{\alpha, r, p} =  \frac1{\sqrt{\alpha!\;r^{|\alpha|^3}}! } \, \cdot (2\mathbb N)^{- \frac{p\alpha}2} \, H_\alpha, \quad \alpha\in \mathcal I, 
\]
forms an orthonormal basis for  $(FS)_{r, p}$. By definition, $(FS)$ is the projective limit of $(FS)_{r, p}$.  If $s\geq r$ and $q> p+1$ then
\begin{align*}
\sum_{\alpha\in \mathcal I} \|K_{\alpha, s, q }\|^2_{(FS)_{r, p}} & = \sum_{\alpha\in \mathcal I} \Big(\frac1{\sqrt{\alpha!s^{|\alpha|^3}}! } \, \cdot (2\mathbb N)^{- \frac{q\alpha}2} \Big)^2 \, \, \cdot \, \|H_\alpha\|^2_{(FS)_{r, p}}\\
& = \sum_{\alpha\in \mathcal I}   \frac1{\alpha!\;s^{|\alpha|^3}! } \, \cdot (2\mathbb N)^{- q\alpha}  \, \cdot \,\alpha!\, r^{|\alpha|^3}! \,\, \,  \, (2\mathbb N)^{p\alpha}\\
& = \sum_{\alpha\in \mathcal I}    \frac{r^{|\alpha|^3}!}{s^{|\alpha|^3}! } \, \cdot (2\mathbb N)^{- (q-p)\alpha}  \\
& \leq \sum_{\alpha\in \mathcal I}     (2\mathbb N)^{- (q-p)\alpha}  < \infty.
\end{align*}
Therefore, the embedding $(FS)_{s, q} \subset (FS)_{r, p}$, $s>r$ and $q> p+1$    is Hilbert--Schmidt 
and we conclude that $(FS)$ is a nuclear space. 
\end{proof}

Since $(FS) = \bigcap_{r\geq 2, p\geq 0} (FS)_{r, p}$ is nuclear, then the inductive limit of the spaces $(FS)_{ -r, -p},$ namely, $(FS)' = \bigcup_{r\geq 2, p\geq 0} (FS)_{-r, -p}$  is nuclear too. We used the property that the  dual of a Fr\' echet space is nuclear if and only if the initial space is nuclear.

\subsubsection{Closedness of $(FS)'$ under the Wick product}
\label{0.2.3}

The \emph{Wick product} of the stochastic functions $f=\sum_{\alpha\in\mathcal I}f_\alpha H_\alpha$ and $g=\sum_{\beta\in\mathcal I}g_\beta H_\beta$ is given by  
\[
f\lozenge g = \sum_{\gamma\in\mathcal I}\sum_{\alpha+\beta=\gamma}f_\alpha g_\beta H_\gamma =\sum_{\alpha\in\mathcal I}\sum_{\beta\leq \alpha} f_\beta g_{\alpha-\beta} H_\alpha.
\] 
In the following theorem we prove that the weighted Hida space $(FS)'$ is closed under the Wick multiplication. 
\begin{theorem}\label{wickclose}
Let the stochastic functions $f\in (FS)_{-r, -p}$ and $g\in (FS)_{-r, -p}$, $r \geq 2$, $p \geq 0$, be given in their chaos expansion forms \[f=\sum_{\alpha\in\mathcal I}f_\alpha H_\alpha \qquad \text{and} \qquad  g=\sum_{\beta\in\mathcal I}g_\beta H_\beta.\] Then, the Wick product $f\lozenge g$ is a well defined element in $(FS)_{- s, -q}$  for $s \geq r \geq 2$ and $q>p$ (i.e. for $q \geq p -1+ k$, where $k>1$).
\end{theorem}

\begin{proof} 
First we note that for $s\geq 2$ and positive numbers $a$ and $b$ it holds that
\begin{equation}\label{fakt nejed}
s^{a^3}!\cdot s^{b^3}!\leq (s^{a^3}+s^{b^3})!\leq s^{a^3+b^3}!\leq s^{(a+b)^3} !  
\end{equation}
which then implies
\[
(s^{(a+b)^3} !)^{-1} \leq (s^{a^3}!)^{-1} \cdot  (s^{b^3}!)^{-1}. 
\]
Thus, in the proof we will use the inequality 
\[ 
(s^{(a+b)^3} !)^{-1} \leq (r^{a^3}!)^{-1} \cdot  (r^{b^3}!)^{-1}, \qquad \text{for} \quad r \leq s .
\] 
Also, we use the estimate
\[ 
(\alpha+\beta)!\leq \alpha!\beta!(2\mathbb N)^{\alpha+\beta},\quad \alpha,\beta\in\mathcal I.
\]
By the Cauchy--Schwartz inequality, for $q=p-1+k$ and $s\geq r \geq 2$ the following holds
\begingroup
\allowdisplaybreaks
\begin{align*}
&\|f\lozenge g\|^2_{(FS)_{-s, -q}}  = \\	
& =\sum_{\gamma\in\mathcal I} \Big\|\sum_{\alpha+\beta=\gamma}f_\alpha g_\beta \Big\|^2 \,\gamma!\,   (s^{|\gamma|^3}!)^{-1} \, (2\mathbb N)^{-q\gamma}\\
& =  \sum_{\gamma\in\mathcal I} \Big\|\sum_{\alpha+\beta=\gamma}f_\alpha g_\beta (\alpha+\beta)!^{\frac12} \Big\|^2 \,    (s^{|\gamma|^3}!)^{-1} \, (2\mathbb N)^{- (p -1+ k)\gamma}\\
& \leq   \sum_{\gamma\in\mathcal I} \Big\|\sum_{\alpha+\beta=\gamma}f_\alpha \,\alpha!^{\frac12} \, (s^{|\alpha|^3}!)^{-\frac 12} \, \, (2\mathbb N)^{- \frac{(p-1)\alpha}2+\frac\alpha2} \, g_\beta \,\beta!^{\frac12}  \, (s^{|\beta|^3}!)^{-\frac 12} \, \, (2\mathbb N)^{- \frac{(p-1) \beta}2+\frac\beta2}\Big\|^2 (2\mathbb N)^{- k\gamma}\\
& \leq  \sum_{\gamma\in\mathcal I} \Big(\sum_{\alpha+\beta=\gamma} \, |f_\alpha|^2 \, \alpha!\, (r^{|\alpha|^3}!)^{- 1} \, \, (2\mathbb N)^{- p\alpha} \Big)  \, \, \Big(\sum_{\alpha+\beta=\gamma} |g_\beta|^2 \, \beta!\, (r^{|\beta|^3}!)^{- 1} \, \, (2\mathbb N)^{- p \beta}\Big) \,    \, (2\mathbb N)^{- k\gamma}\\
& \leq  \, \Big(\sum_{\gamma\in\mathcal I}  (2\mathbb N)^{- k\gamma} \Big) \, \Big(\sum_{\alpha \in \mathcal I} \, |f_\alpha|^2 \,\alpha! \, (r^{|\alpha|^3}!)^{- 1} \, \, (2\mathbb N)^{- p \alpha} \Big)  \, \, \Big(\sum_{\beta\in \mathcal I} |g_\beta|^2 \,\beta! \, (r^{|\beta|^3!})^{- 1} \, \, (2\mathbb N)^{- p \beta}\Big) \,   \\
& \leq M \, \cdot \|f\|^2_{(FS)_{-r, -p}} \, \cdot \, \,  \|g\|^2_{(FS)_{-r, -p}} ,
\end{align*} 
\endgroup
because $M= \sum_{\gamma\in\mathcal I}  (2\mathbb N)^{- k\gamma} < \infty$ for $k>1$. 
\end{proof}

Since the space $(FS)'$ is closed under the Wick multiplication, then its $n$th Wick power is also a well defined element in $(FS)'$. Recall, the $n$th Wick power of $f$ is defined recursively by $f^{\lozenge n} = f\lozenge f^{\lozenge (n-1)}$, for $n\geq 1$ and $f^{\lozenge 0} = 1$. 

\begin{corollary}\label{wick power}
Let  $f\in (FS)_{-r, -p}$, $r\geq 2$, $p\geq 0$. Then its $n$th Wick power $f^{\lozenge n}$, $n\geq 1$ is an element of $(FS)_{-s, - q}$,  for $s\geq r$, $q>p+1$.
\end{corollary}

\begin{proof}
Applying Theorem \ref{wickclose} with $f=g$ we obtain that $f^{\lozenge 2}\in (FS)_{-s,-q}$ for $s\geq r$ and $q=p-1+k$, $k>1$. Now, using the embedding $f\in (FS)_{-r,-p}\subset(FS)_{-s,-p-1+k}$ we obtain $f,f^{\lozenge 2} \in (FS)_{-s,-p-1+k}$ and hence may apply again Theorem \ref{wickclose} to obtain $f^{\lozenge 3}=f\lozenge f^{\lozenge 2}\in (FS)_{-s_1,-q_1}$ for $s_1\geq s$ and $q_1\geq q-1+k_1 = p-2+k+k_1$, where $k_1>1$. Without loss of generality we may let $k_1=k$ and write $f^{\lozenge 3}\in(FS)_{-s,-(p-2+2k)}$. By induction one may easily prove that $f^{\lozenge n}\in(FS)_{-s,-p-(n-1)(k-1)}$. Since $k>1$, we may choose $k=1+\epsilon$, $\epsilon>0$, small enough so that $(n-1)(k-1)<1$, hence $f^{\lozenge n}\in(FS)_{-s,-q}$, for $q>p+1$.
\end{proof}

\subsubsection{Stochastic processes}
\label{0.2.2}

In the same way as it was done in \cite{GRPW} and \cite{ps},  we extend the definition of stochastic processes to processes with the chaos expansion form
\begin{equation}
u(t,\omega)=\sum_{\alpha\in\mathcal I}u_\alpha(t) {H_\alpha}(\omega), 
\label{u ch exp}
\end{equation} 
where the coefficients $u_\alpha$ are elements of the Banach space of functions $C^k[0,T], $ $k\in\mathbb N$. 

\begin{definition}\label{def gen pr}
We say that $u$ given in the chaos expansion form \eqref{u ch exp} is a \emph{generalized stochastic process} in $ C^k[0,T]\otimes (FS)'$ if there exist $r\geq 2$ and $p\geq 0$ such that 
\[
\|u\|_{C^k[0,T]\otimes(FS)'}^2=\sum_{\alpha\in\mathcal I}\alpha!(r^{|\alpha|^3}!)^{-1}\|u_\alpha\|_{C^k[0,T]}^2(2\mathbb N)^{-p\alpha}<\infty .
\]
\end{definition}

The differentiation of a stochastic process can be carried out componentwise in the chaos expansion, i.e. due to the fact that $(FS)'$ is a nuclear space it holds that $C^k([0,T],(FS)')=C^k[0,T]\hat\otimes(FS)'$ where $\hat{\otimes}$ denotes the completion of the tensor product which is the same for the $\varepsilon-$completion  and $\pi-$completion. In the sequel, we will use the notation $\otimes$ instead of $\hat\otimes$.  Hence $C^k[0,T]\otimes (FS)_{-r,-p}$ and $C^k[0,T]\otimes (FS)_{r,p}$ denote subspaces of the corresponding completions. This means that a stochastic process $u(t,\omega)$ is $k$ times continuously differentiable if and only if all of its coefficients $u_\alpha$, $\alpha\in\mathcal I$ are in $C^k[0,T]$.

We keep the same notation when $C^k[0,T]$ is replaced by another Banach space.  The same holds for Banach space valued stochastic processes,  i.e.  elements of  $C^k([0,T],X)\otimes(FS)'$, where $X$ is an arbitrary Banach space. It holds that
\begin{equation}
C^k([0,T],X\otimes (FS)')=C^k([0,T],X)\otimes(FS)'=\bigcup_{r\geq 2, p\geq 0}C^k([0,T],X)\otimes (FS)_{-r,-p}, \quad k\in \mathbb N. 
\label{C^k proc}
\end{equation}

\subsection{Evolution systems}\label{0.1}

We fix the notation and recall some known facts about evolution systems (see \cite[Chapter 5]{Pazy}). Let $X$ be a Banach space. Let $\{A(t)\}_{t\in[s,T]}$ be a family of linear operators in $X$  such that $A(t):D(A(t))\subset X\to X,$ $t\in[s,T]$.  Further, let $f$ be an $X-$valued function $f:[s,T]\to X.$  Consider the initial value problem 
\begin{align}\label{EP}
\frac{d}{dt}u(t)&=A(t)u(t)+f(t),\quad 0\leq s<t\leq T,\\
u(s)&=x.\nonumber
\end{align}
If $u\in C([s,T],X)\cap C^1((s,T],X),$ $u(t)\in D(A(t))$ for all $t\in (s,T]$ and $u$ satisfies \eqref{EP}, then $u$ is a classical solution to \eqref{EP}.

A two parameter family of bounded linear operators $S(t,s),$ $0\leq s\leq t\leq T$ on X is called an evolution system if the following two conditions are satisfied:
\begin{enumerate}
\item $S(s,s)=I$ and $S(t,r)S(r,s)=S(t,s),\quad 0\leq s\leq r\leq t\leq T$; 
\item $(t,s)\mapsto S(t,s)$ is strongly continuous for all $0\leq s\leq t\leq T.$
\end{enumerate}
Clearly, if $S(t,s)$ is an evolution system associated with the homogeneous evolution problem \eqref{EP}, i.e. if $f\equiv 0,$ then a classical solution to \eqref{EP} is given by $ u(t)=S(t,s)x,$ $t\in[s,T]. $

A family $\{A(t)\}_{t\in[s,T]}$ of infinitesimal generators of $C_0-$semigroups on $X$ is called stable if there exist constants $m\geq 1$ and $w\in\mathbb{R}$ (stability constants) such that $(w,\infty)\subseteq \rho(A(t)),$ $t\in[s,T]$ and
\[
\Big\|\prod_{j=1}^kR(\lambda:A(t_j))\Big\|\leq\frac{m}{(\lambda-w)^k},\quad \lambda>w,
\]
for every finite sequence $0\leq s\leq t_1\leq t_2\leq\dots\leq t_k\leq T,\;k=1,2,\dots.$

Let $\{A(t)\}_{t\in [s,T]}$ be a stable family of infinitesimal generators with stability constants $m$ and $w$. Let $B(t),$ $t\in[s,T],$ be a family of bounded linear operators on $X$. If $\|B(t)\|\leq M,$ $t\in [s, T],$ then $\{A(t)+B(t)\}_{t\in[s,T]}$ is a stable family of infinitesimal generators with stability constants $m$ and $w+Mm.$

Let $\{A(t)\}_{t\in [s,T]}$ be a stable family of infinitesimal generators of $C_0-$semigroups on $X$ such that the domain  $D(A(t))=D$ is independent of $t$ and for every $x\in D,$ $A(t)x$ is continuously differentiable in $X.$ If $f\in C^1([s,T],X)$ then for every $x\in D$ the evolution problem \eqref{EP} has a unique classical solution $u$ given by
\[
u(t)=S(t,s)x+\int_s^t S(t,r)f(r)dr,\quad 0\leq s\leq t\leq T.
\]
From the proof of \cite[Theorem 5.3, p. 147]{Pazy} one can obtain
\begin{align*}
\frac{d}{dt}u(t)= A(t)S(t,s)x+A(t)\int_s^t S(t,r)f(r)dr+f(t),\quad s<t\leq T.
\end{align*}
Since $t\mapsto A(t)$ is continuous in $B(D,X)$ and $(t,s)\mapsto S(t,s)$ is strongly continuous for all $0\leq s\leq t\leq T,$ we have additionally that the solution $u$ to \eqref{EP} exhibits the regularity property $u\in C^1([s,T],X)$ and $\frac{d}{dt}u(t)|_{t=s}=A(s)x+f(s).$
Recall that the evolution system $S(t,s)$ satisfies:
\begin{enumerate}
\item $\|S(t,s)\|\leq me^{w(t-s)},\ 0\leq s\leq t\leq T;$
\item $S(t,s)D\subseteq D;$
\item $S(t,s)x$ is continuous in $D$ for all $0\leq s\leq t\leq T$ and $x\in D.$
\end{enumerate}

\begin{remark}
Considering infinitesimal generators depending on $t,$ we follow the standard  approach of Yosida (cf. \cite{51}, \cite{23}). We refer to \cite{NZ} for a method based on an equivalent  operator extension problem (see also references in \cite{NZ}). The chaos expansion approach, which is the essence of our paper, requires the existence results for the propagator system i.e. for the coordinate-wise deterministic Cauchy problems. For this purpose we demonstrate the applications of the hyperbolic Cauchy problem given in \cite{Pazy}.
\end{remark}

\subsection{Entire functions on Banach algebras}\label{0.3}

In this subsection we fix some notation and provide some basic properties of entire functions on Banach spaces that will be used in the sequel (see \cite{blum}).

Let $X$ be a Banach algebra with the uniform topology. Continuity of a function $\Phi,$ defined on a region $Y\subset X$ with values in $X,$ is defined as usual. 

\begin{definition}
A function $\Phi$ has a derivative $\Phi'(u_\mathbf{0})$ at $u_\mathbf{0}\in X$ if for every $\varepsilon>0$ a $\delta>0$ can be found such that for all $u\in Y$ with $\|u-u_\mathbf{0}\|_X<\delta,$
\[
\left\|\Phi(u)-\Phi(u_\mathbf{0})-(u-u_\mathbf{0})\Phi'(u_\mathbf{0})\right\|_X< \varepsilon \|u-u_\mathbf{0}\|_X.
\]
If $\Phi$ has a derivative throughout a neighbourhood of $u_\mathbf{0}$ then $\Phi$ is said to be analytic at $u_\mathbf{0}.$ If $\Phi$ is analytic in the whole space $X$ then it is said to be entire.
\end{definition}

As presented in \cite{blum}, if $\Phi$ is analytic at $u_\mathbf{0}\in Y$ then one can establish its Taylor expansion, i.e. there exists an infinite sequence $a_n\in X,\;n\in\mathbb{N}$ such that
\[
\Phi(u)=\sum_{n=0}^\infty a_n (u-u_\mathbf{0})^n,\quad u\in Y,\;\|u-u_\mathbf{0}\|_X<R,
\]
where $R$ is the ratio of convergence and $a_n=\frac{\Phi^{(n)}(u_\mathbf{0})}{n!}.$ The ratio of convergence is given by $1/R=\lim_{n\to \infty}\sqrt[n]{\|a_n\|_X}.$ Hence, if $\Phi$ is an entire function on $X$ then $\lim_{n\to \infty}\sqrt[n]{\|a_n\|_X}=0$ and one obtains
\[
\Phi(u)=\sum_{n=0}^\infty a_n u^n \qquad\mbox{and}\qquad \Phi^{(k)}(u)=\sum_{n=0}^\infty a_{n+k} u^n,\quad u\in X,\;k\in \mathbb{N}.
\]

In the sequel, we will use the notation $\varphi$ for the function obtained from the entire function $\Phi$ as follows
\[
\varphi(\xi):=\sum_{n=0}^\infty \|a_n\|_X \; \xi^n,\quad \xi\in \mathbb{R}.
\]
Since $\lim_{n\to \infty}\sqrt[n]{\|a_n\|_X}=0,$ the function $\varphi$ is also analytic on the whole $\mathbb{R}.$ Recall, analyticity of the function $\varphi$ implies that for every compact set $K\subset \mathbb{R}$ there exists a constant $C>0$ such that for every $k\in \mathbb{N}$ the following bound holds 
\begin{equation}\label{const C}
\sup_{\xi \in K}|\varphi^{(k)}(\xi)|\leq C^k k!.
\end{equation}

Note here that if $u_\mathbf{0}(t),$ $t\in[0,T]$ is an $X-$valued function of a certain regularity with respect to $t$, e.g. if $u_\mathbf{0}\in C^1([0,T],X),$ then by the chain rule theorem the function $\Phi(u_\mathbf{0}(t))=\sum_{n=0}^\infty a_n u^n_\mathbf{0}(t),$ $t\in [0,T]$ is of the same regularity.

In what follows, we use the notation $\Phi^{\lozenge}$ for a function defined on the $X-$valued weighted Hida spaces $X\otimes (FS)',$ generated from the entire function $\Phi$ on $X,$ in the following way
\[
\Phi^{\lozenge}(u)=\sum_{n=0}^\infty a_n u^{\lozenge n},\quad u\in X\otimes (FS)',
\]
where $a_n\in X,\;n\in \mathbb{N}$ and $\lim_{n\to \infty}\sqrt[n]{\|a_n\|_X}=0.$ Since the classical product of generalized random variables is not well defined in $X\otimes (FS)',$ here we interchange the classical product by the Wick product defined in the previous Subsection \ref{0.2}. From Corollary \ref{wick power} one can infer that if $u\in X\otimes (FS)'$ then $\Phi^{\lozenge}(u),$ expressed in this way, is a well defined element of $X\otimes (FS)'.$

\section{Stochastic nonlinear evolution equations}
\label{1.0}

We consider the equation \eqref{PNLJ}, i.e. the equation of the form:

\begin{align}\label{NLJ}
u_t (t, \omega)&= \mathbf A \, u(t, \omega) + \sum_{n=0}^\infty a_n u^{\lozenge n}(t, \omega)+f(t,\omega), \quad t\in (0,T], \quad \omega\in \Omega\\\nonumber
u(0,\omega) &= u^0(\omega), \quad \omega\in\Omega.
\end{align} 
We are looking for a solution to \eqref{NLJ} as an $X$-valued stochastic process $u(t)\in X\otimes (FS)',\;t\in[0,T]$ represented in the form 
\begin{equation}\label{proces}
 u(t,\omega)= \sum_{\alpha\in \mathcal I} u_\alpha(t) \,\,  H_\alpha(\omega),\quad t\in[0,T],\quad \omega\in \Omega.
 \end{equation}
Let $\mathbf A: \mathbb{D}\subset X\otimes (FS)' \to X\otimes (FS)'$ be a coordinatewise operator that corresponds to a family of deterministic operators $A_\alpha: \, D_\alpha\subset X \to X$, $\alpha\in \mathcal{I}$ 
 \[
\mathbf{A}\, u(t,\omega)=\mathbf{A}\left(\sum_{\alpha\in \mathcal I} u_\alpha(t) \, H_\alpha(\omega)\right)=\sum_{\alpha\in \mathcal I} A_\alpha u_\alpha(t) \,\,  H_\alpha(\omega),\quad u\in \mathbb{D},
\]
(see \cite[Section 2]{Milica}).  The chaos expansion representation of the Wick-square is given by
\begin{align*}
u^{\lozenge 2}(t,\omega)&= \sum_{\alpha \in \mathcal I} \Big( \sum_{\gamma\leq \alpha } \, u_\gamma(t) \,\, u_{\alpha-\gamma} (t)\Big) \, H_\alpha(\omega)\\\nonumber
&=u^2_{\mathbf{0}}(t)\,H_{\mathbf{0}}(\omega)+\sum_{|\alpha|>0} \Big(2u_{\mathbf{0}}(t)\,u_\alpha(t)+ \sum_{0<\gamma< \alpha } \, u_\gamma(t) \,\, u_{\alpha-\gamma} (t)\Big) \, H_\alpha(\omega),
\end{align*}
where $t\in[0,T],$ $\omega\in\Omega.$ Let $u^{\lozenge m}_\alpha(t)$, $\alpha \in \mathcal I$, $m\in \mathbb N$ denote the coefficients of the chaos expansion of the $m$th Wick power, i.e.  $u^{\lozenge m}(t,\omega) = \sum_{\alpha\in \mathcal I} u^{\lozenge m}_\alpha (t) H_\alpha (\omega)$, for $m\in \mathbb N$. Then, for arbitrary $n\in \mathbb{N}$, one can infer  that the $n$th Wick-power is given by
\begin{align}\label{WPn}
&  u^{\lozenge n}(t,\omega)=u^{\lozenge n-1}(t,\omega)\lozenge u(t,\omega) =\sum_{\alpha \in \mathcal I} \Big( \sum_{\gamma\leq \alpha } \, u^{\lozenge n-1}_\gamma(t) \,\, u_{\alpha-\gamma} (t)\Big) \, H_\alpha(\omega)
\\\nonumber
&=u^n_{\mathbf{0}}(t)\,H_{\mathbf{0}}(\omega)+\sum_{|\alpha|>0} \Bigg(\binom{n}{1}u_{\mathbf{0}}^{n-1}(t)\,u_\alpha(t)+\binom{n}{2}u_{\mathbf{0}}^{n-2}(t) \sum_{0<\gamma_1< \alpha } \, u_{\alpha-\gamma_1}(t) \,\, u_{\gamma_1} (t)\\\nonumber
&+ \binom{n}{3}u_{\mathbf{0}}^{n-3} (t)\sum_{0<\gamma_1< \alpha } \sum_{0<\gamma_2< \gamma_1 }\, u_{\alpha-\gamma_1}(t) \,\, u_{\gamma_1-\gamma_2} (t) u_{\gamma_2}(t)+\dots+\\\nonumber
&+ \binom{n}{n}\sum_{0<\gamma_1< \alpha } \sum_{0<\gamma_2< \gamma_1 }\dots \sum_{0<\gamma_{n-1}< \gamma_{n-2} }u_{\alpha-\gamma_1}(t) \, u_{\gamma_1-\gamma_2} (t)\dots  u_{\gamma_{n-2}-\gamma_{n-1}}(t)u_{\gamma_{n-1}}(t)\Bigg) H_\alpha(\omega),
\end{align}
where $t\in[0,T],$ $\omega\in\Omega.$ Moreover, we recall that the Wick power $u^{\lozenge n}$ of a stochastic process $u\in X\otimes (FS)_{-r, -p}$ is an element of $X\otimes (FS)_{-r, -q}$, for $q> p+1$, see Corollary \ref{wick power}.

Now, the nonlinear part of \eqref{NLJ}  is of the form
\begin{align*}
\sum_{n=0}^\infty & a_n u^{\lozenge n}(t,\omega)  = a_0 + a_1 u(t,\omega) + a_2 u^{\lozenge 2}(t,\omega) + a_3 u^{\lozenge 3}(t,\omega) + 
\cdots + a_n u^{\lozenge n}(t,\omega) + \cdots
\end{align*}
which, after using \eqref{WPn}, becomes 
\begingroup
\allowdisplaybreaks
\begin{align*}
& = a_0 + a_1 \Big[u_\mathbf{0}(t)\,H_\mathbf{0}(\omega) +\sum_{|\alpha|>0}u_\alpha(t)\,H_\alpha (\omega)\Big]\\
& + a_2\, \Big[ u^2_{\mathbf{0}}(t)\,H_{\mathbf{0}}(\omega)+\sum_{|\alpha|>0} \Big(2 u_{\mathbf{0}}(t)\,u_\alpha(t)+ \sum_{0<\gamma< \alpha } \, u_\gamma(t) \, u_{\alpha-\gamma} (t)\Big) \, H_\alpha(\omega)\Big]\\
& + a_3 \, \Bigg[ u^3_{\mathbf{0}}(t)\,H_{\mathbf{0}}(\omega)+\sum_{|\alpha|>0} \Bigg(3 u_{\mathbf{0}}^{2}(t)\,u_\alpha(t)+ 3 u_{\mathbf{0}}(t) \sum_{0<\gamma_1< \alpha } \, u_{\alpha-\gamma_1}(t) \, u_{\gamma_1} (t)\\
&\hspace{4cm} + \sum_{0<\gamma_1< \alpha } \sum_{0<\gamma_2< \gamma_1 }\, u_{\alpha-\gamma_1}(t) \,u_{\gamma_1-\gamma_2} (t) u_{\gamma_2}(t)\Bigg) \, H_\alpha(\omega)  \Bigg]\\
& + a_4 \,\Bigg[u^4_{\mathbf{0}}(t)\,H_{\mathbf{0}}(\omega)+\sum_{|\alpha|>0} \Bigg( 4 u_{\mathbf{0}}^{3}(t)\,u_\alpha(t)+ 6 u_{\mathbf{0}}^{2}(t) \sum_{0<\gamma_1< \alpha } \, u_{\alpha-\gamma_1}(t) \, u_{\gamma_1} (t)\\
&\hspace{3.3cm} +  4 u_{\mathbf{0}}(t) \sum_{0<\gamma_1< \alpha } \sum_{0<\gamma_2< \gamma_1 }\, u_{\alpha-\gamma_1}(t) \, u_{\gamma_1-\gamma_2} (t) u_{\gamma_2}(t)\\
&\hspace{3.3cm} + \sum_{0<\gamma_1< \alpha } \sum_{0<\gamma_2< \gamma_1 }\sum_{0<\gamma_{3}< \gamma_{2} }\, u_{\alpha-\gamma_1}(t) \, u_{\gamma_1-\gamma_2} (t) u_{\gamma_{2}-\gamma_{3}}(t)u_{\gamma_{3}}(t)\Bigg) \, H_\alpha(\omega) \Bigg] \\
& + \cdots +\\
& + a_n \, \Bigg[ u^n_{\mathbf{0}}(t)\,H_{\mathbf{0}}(\omega)+\sum_{|\alpha|>0} \Bigg(\binom{n}{1}u_{\mathbf{0}}^{n-1}(t)\,u_\alpha(t)+\binom{n}{2}u_{\mathbf{0}}^{n-2} (t) \sum_{0<\gamma_1< \alpha } \, u_{\alpha-\gamma_1}(t) \, u_{\gamma_1} (t)\\\nonumber
& \hspace{4cm} + \binom{n}{3}u_{\mathbf{0}}^{n-3}(t) \sum_{0<\gamma_1< \alpha } \sum_{0<\gamma_2< \gamma_1 }\, u_{\alpha-\gamma_1}(t) \, u_{\gamma_1-\gamma_2} (t) u_{\gamma_2}(t)+\dots+\\
& \hspace{4cm} + \binom{n}{n}\sum_{0<\gamma_1< \alpha } \sum_{0<\gamma_2< \gamma_1 }\dots \sum_{0<\gamma_{n-1}< \gamma_{n-2} }\,\\
& \hspace{5.8cm}  u_{\alpha-\gamma_1}(t) \, u_{\gamma_1-\gamma_2} (t)\dots  u_{\gamma_{n-2}-\gamma_{n-1}}(t)u_{\gamma_{n-1}}(t)\Bigg) \, H_\alpha(\omega) \Bigg] \\
& + \cdots\\
& = \Phi (u_\mathbf{0}(t))\,H_\mathbf{0}(\omega) + \sum_{|\alpha|>0} \Bigg(\Phi'(u_\mathbf{0}(t))\,u_\alpha (t) + \frac{\Phi''(u_\mathbf{0}(t))}{2!} \sum_{0<\gamma_1< \alpha } \, u_{\alpha-\gamma_1}(t) \, u_{\gamma_1} (t) \\
& \hspace{4cm}  + \frac{\Phi^{(3)}(u_\mathbf{0}(t))}{3!}  \sum_{0<\gamma_1< \alpha } \sum_{0<\gamma_2< \gamma_1 }\, u_{\alpha-\gamma_1}(t) \, u_{\gamma_1-\gamma_2} (t) u_{\gamma_2}(t) + \cdots\\
& \hspace{4cm}  + \frac{\Phi^{(|\alpha|)}(u_\mathbf{0}(t))}{|\alpha|!} \sum_{0<\gamma_1< \alpha } \sum_{0<\gamma_2< \gamma_1 }\dots \sum_{0<\gamma_{|\alpha|-1}< \gamma_{|\alpha|-2} }\,\\
& \hspace{5.5cm}  u_{\alpha-\gamma_1}(t) u_{\gamma_1-\gamma_2} (t)\dots  u_{\gamma_{|\alpha|-2}-\gamma_{|\alpha|-1}}(t)u_{\gamma_{|\alpha|-1}}(t)\Bigg)H_\alpha (\omega).
\end{align*}
\endgroup
Finally,
\[
\sum_{n=0}^\infty a_n u^{\lozenge n}(t,\omega)  = \Phi \big(u_\mathbf{0}(t)\big)\,H_{\mathbf{0}}(\omega)+\sum_{|\alpha|>0} \Bigg(\Phi'\big(u_\mathbf{0}(t)\big)\,u_\alpha(t)+r_{\alpha,\Phi} (t)\Bigg) \, H_\alpha(\omega)
\]
where $t\in[0,T],$ $\omega\in\Omega.$ For later use, let us emphasize here that for $\alpha\in \mathcal{I},$ $|\alpha|=1$ the functions $r_{\alpha,\Phi}(t)=0,$ $t\in [0,T]$ and for $\alpha\in\mathcal{I},$ $|\alpha|\geq 2$ the functions $r_{\alpha,\Phi}$ are of the form
\begin{align}\label{ostatak}
r_{\alpha,\Phi}(t)=\sum_{k=2}^{|\alpha|}\frac{\Phi^{(k)}(u_{\mathbf{0}}(t))}{k!}\sum_{\mathbf{0}<\gamma_1<\alpha}  \sum_{\mathbf{0}<\gamma_2<\gamma_1} &\dots \sum_{\mathbf{0}<\gamma_{k-1}<\gamma_{k-2}}\\\nonumber
& u_{\alpha-\gamma_1}(t)u_{\gamma_1-\gamma_2}(t)\dots u_{\gamma_{k-2}-\gamma_{k-1}}(t) u_{\gamma_{k-1}}(t),
\end{align} 
$t\in[0,T],$ so they contain only the coordinate functions $u_\beta,$ $\beta< \alpha.$

We rewrite all processes that figure in \eqref{NLJ} in their corresponding  Wiener-It\^o chaos expansion form and obtain 
\[
\begin{split}
\sum_{\alpha \in \mathcal I} \, \frac{d}{dt}u_\alpha(t) \, H_\alpha(\omega) &= \Big(A_\mathbf{0}u_\mathbf{0}(t)+ \Phi \big(u_\mathbf{0}(t)\big)+ f_\mathbf{0}(t)\Big)\,H_\mathbf{0}(\omega) \\
&+ \sum_{|\alpha|>0}\Bigg( A_\alpha\,u_\alpha(t) + \Phi'\big(u_\mathbf{0}(t)\big)\,u_\alpha(t)+r_{\alpha,\Phi} (t) + f_\alpha(t)\Bigg) \, H_\alpha(\omega) \\
\sum_{\alpha \in \mathcal I} u_\alpha(0) \ H_\alpha(\omega) & = \sum_{\alpha \in \mathcal I} u^0_\alpha \ H_\alpha(\omega).
\end{split}
\] 
Due to the orthogonality of the basis $\{H_\alpha\}_{\alpha\in \mathcal I}$ this reduces to the system of infinitely many deterministic Cauchy problems: 
\begin{enumerate}
\item[$1^\circ$] for $\alpha =\mathbf{0}$ 
\begin{equation}\label{nelinearna det}
\frac{d}{dt} u_{\mathbf{0}} (t) =  A_{\mathbf{0}} u_{\mathbf{0}}  (t) +   \Phi \big(u_\mathbf{0}(t)\big) +f_\mathbf{0}(t), \quad u_{\mathbf{0}}(0) = u_{\mathbf{0}}^0,  \qquad \text{and}
\end{equation}
\item[$2^\circ$] for  $\alpha >\mathbf{0}$
\begin{equation}\label{sistem 2}
\frac{d}{dt}u_\alpha (t)=  \big( A_\alpha +  \Phi'\big(u_\mathbf{0}(t)\big) \,Id \big) \, u_\alpha(t)  +  r_{\alpha,\Phi} (t) + f_\alpha(t), \quad  u_\alpha (0) =  u_\alpha^0 ,
\end{equation} 
with $t\in (0,T]$ and $\omega\in\Omega$. 
\end{enumerate}
Let
\[
B_{\alpha,\Phi}(t) = A_\alpha +  \Phi'\big(u_\mathbf{0}(t)\big)\, Id \qquad  \text{and} \qquad g_{\alpha,\Phi} (t)= r_{\alpha,\Phi} (t) + f_\alpha(t), \quad t\in(0,T]
\] 
for all $\alpha > \mathbf{0}$. Then, the system \eqref{sistem 2} can be written in the form 
\begin{equation}\label{sistem 3}
\frac{d}{dt}u_\alpha (t) =  B_{\alpha,\Phi}(t)  \, u_\alpha(t) +  g_{\alpha,\Phi}(t), \quad t\in(0,T];\qquad u_\alpha (0) =  u_\alpha^0 , \quad \alpha> \mathbf{0}. 
\end{equation}
Note that the inhomogeneous part $g_{\alpha,\Phi}$ in \eqref{sistem 3} does not contain any of the functions $u_\beta,\;\beta<\alpha$ for $|\alpha|=1$, while for $|\alpha|>1$ it involves also $u_\beta,\;\beta<\alpha$. Hence, we distinguish these two cases.
\begin{enumerate}
\item[(a)] 
Let $|\alpha|=1$, i.e. $\alpha=\varepsilon_k$, $k\in \mathbb N.$ Then $g_{\varepsilon_k,\Phi} =f_{\varepsilon_k}$, $k\in \mathbb N$ and thus \eqref{sistem 3} transforms to 
\begin{equation}\label{det jed duzina 1}
\frac{d}{dt}u_{\varepsilon_k} (t)=  B_{\varepsilon_k,\Phi}(t)  \, u_{\varepsilon_k} (t) + f_{\varepsilon_k}(t),\quad t\in(0,T];  \qquad  u_{\varepsilon_k} (0) =  u_{\varepsilon_k}^0.
\end{equation}
\item[(b)] Let $|\alpha|>1.$ Then 
\[
\frac{d}{dt}u_\alpha (t) =  B_{\alpha,\Phi}(t)  \, u_\alpha(t) +  g_{\alpha,\Phi}(t), \quad t\in(0,T];\qquad  u_\alpha (0) =  u_\alpha^0.
\]
\end{enumerate}

Each solution $u$ to \eqref{NLJ} can be represented in the form \eqref{proces} and hence its coefficients $u_{\mathbf{0}} $ and $u_\alpha$  for  $\alpha> \mathbf{0}$ must satisfy   \eqref{nelinearna det} and  \eqref{sistem 3} respectively. Vice versa, if the coefficients $u_{\mathbf{0}} $ and $u_\alpha$  for  $\alpha> \mathbf{0}$ solve \eqref{nelinearna det} and  \eqref{sistem 3} respectively, and if for every $t\in [0, T]$ the series in \eqref{proces} represented by these coefficients exists in $X\otimes (FS)'$, then it defines a solution to \eqref{NLJ}.

\begin{definition} 
An $X-$valued generalized stochastic process $u(t)=\sum_{\alpha\in\mathcal{I}}u_\alpha(t)H_\alpha\in X\otimes (FS)',\;t\in[0,T]$ is a coordinatewise classical solution to \eqref{NLJ} if $u_\mathbf{0}$ is a classical solution to \eqref{nelinearna det} and for every $\alpha\in\mathcal{I}\setminus\{\mathbf{0}\},$ the coefficient $u_\alpha$ is a classical solution to \eqref{sistem 3}. The coordinatewise  solution $u(t)\in X\otimes (FS)',\;t\in[0,T]$ is an almost classical solution to \eqref{NLJ} if $u\in C([0,T], X \otimes (FS)').$ An almost classical solution is a classical solution if $u\in C([0,T], X \otimes (FS)') \cap C^1((0,T], X\otimes (FS)')$.
\end{definition}

We assume that the following hold:
\begin{enumerate}
\item[(A1)] The operators  $A_\alpha,\;\alpha\in \mathcal{I}$ are infinitesimal generators of $C_0-$semigroups $\{T_\alpha (s)\}_{s\geq 0}$ with a common domain $D_{\alpha}=D,\;\alpha\in \mathcal{I}$  dense in $X$.  We assume that there exist constants $m\geq 1$ and $w\in \mathbb{R}$ such that
\[
\|T_\alpha(s)\|\leq me^{w s},\;s\geq 0 \quad\mbox{for all}\quad\alpha\in \mathcal{I}.
\]
The action of  $\mathbf A$ is given by 
\[
\mathbf A(u)=\sum_{\alpha\in\mathcal I}A_\alpha(u_\alpha)H_\alpha,
\] 
for $u \in \mathbb{D}\subseteq D\otimes (FS)'$ of the form \eqref{proces},  where
\begin{align*}
\mathbb{D}=\Big\{u= \sum_{\alpha\in\mathcal I}u_\alpha \, \,  H_\alpha & \in D\otimes (FS)':\\
& \exists r\geq 2,\;\exists p\geq 0,\; \sum_{\alpha\in\mathcal I}\|A_\alpha(u_\alpha)\|^2_X \alpha!(r^{|\alpha|^3}!)^{-1} (2\mathbb N)^{-p\alpha}<\infty\Big\}.
\end{align*}
 			
\item[(A2)] The initial value $u^0=\sum_{\alpha\in\mathcal{I}}u^0_\alpha H_\alpha\in\mathbb{D}$, i.e. $u_\alpha^0\in D$ for every $\alpha\in\mathcal{I}$ and there exist $r\geq 2$ and  $p\geq 0$ such that 
\begin{align*}
\|u^0\|^2_{X \otimes(FS)_{-r,-p}} & =\sum_{\alpha\in \mathcal{I}}\|u_\alpha^0\|_X^2 \alpha!(r^{|\alpha|^3}!)^{-1} (2\mathbb N)^{-p \alpha}<\infty, \\
\|\mathbf A u^0\|^2_{X \otimes(FS)_{-r,-p}} & =\sum_{\alpha\in \mathcal{I}}\|A_\alpha (u_\alpha^0)\|_X^2 \alpha!(r^{|\alpha|^3}!)^{-1} (2\mathbb N)^{-p \alpha}<\infty.
\end{align*}
 		
\item[(A3)] The inhomogeneous part $f(t,\omega)=\sum_{\alpha\in\mathcal{I}}f_\alpha(t)H_\alpha(\omega),\;t\in[0,T],\;\omega\in \Omega$ belongs to $ C^1([0,T],X)\otimes(FS)';$ hence $t\mapsto f_\alpha(t)\in C^1([0,T],X),\;\alpha\in \mathcal{I}$ and there exist $r\geq 2$ and $p\geq 0$ such that
\begin{align*}
\|f\|&^2_{C^1([0,T],X) \otimes(FS)_{-r,-p}}  =\sum_{\alpha\in \mathcal{I}}\|f_\alpha\|^2_{C^1([0,T],X)} \alpha!(r^{|\alpha|^3}!)^{-1} (2\mathbb{N})^{-p\alpha}  \\
& =\sum_{\alpha\in \mathcal{I}}\Big(\sup_{t\in[0,T]}\|f_\alpha(t)\|_X+\sup_{t\in[0,T]}\|f'_\alpha(t)\|_X\Big)^2 \alpha!(r^{|\alpha|^3}!)^{-1} (2\mathbb{N})^{-p\alpha}<\infty.
\end{align*} 		
 			
\item[(A4)] The Cauchy problem
\[
\frac{d}{dt} u_{\mathbf{0}} (t) =  A_{\mathbf{0}} u_{\mathbf{0}}  (t) +   \Phi \big(u_\mathbf{0}(t)\big) + f_\mathbf{0}(t) ,\quad t\in(0,T]; \quad u_{\mathbf{0}}(0) = u_{\mathbf{0}}^0,
\]
has a unique classical solution $u_{\mathbf{0}}\in C^1([0,T],X).$ 
\end{enumerate}

Now we focus on solving \eqref{sistem 3} for $\alpha > \mathbf{0}$. 

\begin{lemma}\label{Lema} Let the assumptions (A1)-(A4) be fulfilled. Then for every $\alpha>\mathbf{0}$ the evolution system \eqref{sistem 3} has a unique classical solution $u_\alpha\in C^1([0,T],X).$ 
\end{lemma}

\begin{proof}
First, for every $\alpha> \mathbf{0},$ we consider the family of operators $B_{\alpha,\Phi}(t) =  A_\alpha + \Phi'\big(u_\mathbf{0}(t)\big) Id,$ $t\in[0,T].$  According to assumption (A1), the constant family $\{A_\alpha(t)\}_{t\in[0,T]}=\{A_\alpha\}_{t\in[0,T]}$ is a stable family of infinitesimal generators of a $C_0-$semigroup $\{T_\alpha(s)\}_{s\geq 0}$ on $X$ satisfying $\|T_\alpha(s)\|\leq me^{w s}$  with stability constants $m\geq 1$ and $w\in\mathbb{R}.$ Let 
\begin{equation}\label{ocena u^0}
M_\Phi=\sup_{t\in[0,T]}\|u_\mathbf{0}(t)\|_X.
\end{equation} 
The perturbation $\Phi'\big(u_\mathbf{0}(t)\big)\,Id:X\to X,$ $t\in[0,T]$ is a family of uniformly bounded linear operators such that
\begin{align*}
\|\Phi'\big(u_\mathbf{0}(t)\big)x\|_X & =\|\Phi'\big(u_\mathbf{0}(t)\big)\|_X\|x\|_X = \left\| \sum_{n=0}^\infty a_{n+1} u^n_\mathbf{0}(t) \right\|_X \|x\|_X \\
&\leq \left( \sum_{n=0}^\infty \|a_{n+1}\|_X\|u_\mathbf{0}(t)\|^n_X \right)\|x\|_X  \leq \left( \sum_{n=0}^\infty \|a_{n+1}\|_X M^n_\Phi \right) \|x\|_X  \\
& \leq \varphi'(M_\Phi)\|x\|_X, \quad x\in X,\quad t\in[0,T],
\end{align*}
i.e. $\|\Phi'\big(u_\mathbf{0}(t)\big)Id\|\leq \varphi'(M_\Phi),$ $t\in [0,T].$ Here we have used the notation introduced in Subsection \ref{0.3}.  Thus, for every $\alpha>\mathbf{0},$ the family $\{A_\alpha+\Phi'\big(u_\mathbf{0}(t)\big)Id\}_{t\in[0,T]}$ is a stable family of infinitesimal generators with stability constants $m$ and $w+\varphi'(M_\Phi) m.$ By assumption (A4) the function $u_\mathbf{0}\in C^1([0,T],X)$ so we obtain continuous differentiability of  $(A_\alpha+\Phi'\big(u_\mathbf{0}(t)\big)Id)x,$ $t\in[0,T]$ for every $x\in D$ and for every $\alpha>\mathbf{0}.$ Additionally, the domain of the operators $\Phi'\big(u_\mathbf{0}(t)\big)Id$ is the entire space $X$ which implies that all of the operators $B_{\alpha,\Phi}(t),$ $t\in[0,T]$ have a common domain $D(B_{\alpha,\Phi}(t))=D(A_\alpha)=D$ not depending on $t.$ Notice here that assumption (A1) additionally provides the same domain $D$ of the family $\{B_{\alpha,\Phi}(t)\}_{t\in[0,T]}$ for all $\alpha>\mathbf{0}.$

Finally, one can associate the unique evolution system $ S_{\alpha,\Phi}(t,s)$,  for $0 \leq s \leq t\leq T$ for all $\alpha> \mathbf{0}$ to the system \eqref{sistem 3} such that 
\begin{equation}\label{ocena S_alpha}
\|S_{\alpha,\Phi}(t,s)\| \leq  me^{w_\Phi \, (t-s)} \leq me^{w_\Phi (T-s)}, \quad  0 \leq s \leq t\leq T, \quad \alpha> \mathbf{0},
\end{equation}
where $w_\Phi=w+\varphi'(M_\Phi) m$, see \cite[Thm 4.8., p. 145]{Pazy}. Without loss of generality we may assume that $w>0$ and thus will be $w_\Phi>0$.

Now one can solve the infinite system of the Cauchy problems \eqref{sistem 3} by induction on the length of the multiindex $\alpha$. Let $|\alpha|=1.$ Since $f_{\varepsilon_k}\in C^1([0,T],X),$ we obtain the unique classical  solution $u_{\varepsilon_k}\in C^1([0,T],X)$ to \eqref{det jed duzina 1} given by
\begin{equation}\label{sol hom}
u_{\varepsilon_k} (t)= S_{\varepsilon_k,\Phi} (t,0) \, u_{\varepsilon_k}^0 + \int_0^t \, S_{\varepsilon_k,\Phi} (t,s) \, f_{\varepsilon_k}(s) \, ds, \quad t\in [0,T].
\end{equation}
Now let for every $\beta\in \mathcal{I}$ such that $\mathbf{0}<\beta<\alpha$ the unique classical solution to \eqref{sistem 3} satisfies  $u_{\beta}\in C^1([0,T],X).$ Then for fixed $|\alpha|>1$ the inhomogeneous part $g_{\alpha,\Phi}\in C^1([0,T],X)$ and the solution to \eqref{sistem 3}  is of the form
\begin{equation}\label{sol nehom}
u_{\alpha}  (t) = S_{\alpha,\Phi} (t,0) \, u_{\alpha}^0   + \int_0^t \, S_{\alpha,\Phi} (t,s) \, g_{\alpha,\Phi}(s) \, ds , \quad t\in [0,T],
\end{equation}
where $u_{\alpha}\in C^1([0,T],X).$ For more details see \cite[Thm 5.3., p. 147]{Pazy}. 
\end{proof}

We proceed with one technical lemma. 

\begin{lemma} \label{N(alfa, k)2} Let $\alpha\in \mathcal{I}\setminus \{\mathbf{0}\}$ be given and $k\in \mathbb{N},$ $k\leq |\alpha|.$ Denote by $N(\alpha,k)$  the number of all possible combinations in which a multiindex $\alpha$ can be written as a sum of $k$ strictly smaller and nonzero multiindices. Then
\[
N(\alpha,k)\leq 2^{k|\alpha|},\qquad\alpha\in \mathcal{I}\setminus \{\mathbf{0}\},\quad k\in\mathbb{N},\;k\leq |\alpha|.
\]
\end{lemma}

\begin{proof}
Here we will use a known combinatorial result about the number of partitions of a positive integer $n$ into $k$ parts.  Namely, a positive integer $n\in \mathbb{N}$ can be written as a sum of $k$ integers, i.e. in the form $n=x_1+\dots+x_k,$ $x_i\in \{0,1,\dots, n\},$ in $\binom{n+k-1}{n}=\binom{n+k-1}{k-1}$ different ways. 

Let $\alpha\in \mathcal{I}\setminus \{\mathbf{0}\}$ be written as a sum of $k$ multiindices $\gamma^1,\dots,\gamma^k\in \mathcal{I}\setminus\{\mathbf{0}\}.$ Then 
\[
(\alpha_1,\alpha_2,\dots, \alpha_s,0,0,\dots)=(\gamma^1_1,\gamma^1_2,\dots,\gamma^1_s,0,0,\dots)+\dots+(\gamma^k_1,\gamma^k_2,\dots,\gamma^k_s,0,0,\dots),
\]
i.e.  each component of $\alpha$ can be decomposed 
\begin{align*}
& \alpha_1=\gamma^1_1+\dots+\gamma^k_1,\quad \gamma^i_1\in\{0,\dots, \alpha_1\},\;i\in \{1,\dots,k\}\quad \mbox{in}\quad \binom{\alpha_1+k-1}{k-1} \quad\mbox{different ways;}\\
& \alpha_2=\gamma^1_2+\dots+\gamma^k_2,\quad \gamma^i_2\in\{0,\dots, \alpha_2\},\;i\in \{1,\dots,k\}\quad \mbox{in}\quad \binom{\alpha_2+k-1}{k-1} \quad\mbox{different ways;}\\
& ...\\
& \alpha_s=\gamma^1_s+\dots+\gamma^k_s,\quad \gamma^i_s\in\{0,\dots, \alpha_s\},\;i\in \{1,\dots,k\}\quad \mbox{in}\quad \binom{\alpha_s+k-1}{k-1} \quad\mbox{different ways.}
\end{align*}
Hence, $N(\alpha,k)\leq \prod_{j=1}^{s}\binom{\alpha_j+k-1}{k-1},$ since by this product we also counted for partitions that include the $\mathbf{0}$ multiindex. Finally, we obtain 
\[
N(\alpha,k)\leq \prod_{j=1}^s\binom{\alpha_j+k-1}{k-1}\leq \prod_{j=1}^s 2^{\alpha_j+k-1}=2^{s(k-1)}\prod_{j=1}^s 2^{\alpha_j} = 2^{s(k-1)}2^{|\alpha|}\leq 2^{k|\alpha|},
\]
because $s\leq |\alpha|,$ since $s$ is actually the number of nonzero coordinates of $\alpha\in \mathcal{I}\setminus\{\mathbf{0}\}.$ 
\end{proof}

\subsection{Proof of the main theorem}\label{1.1}

The statement of the main theorem is as follows.

\begin{theorem}\label{main}
Let the assumptions (A1)-(A4) be fulfilled.  Then there exists a unique almost classical solution $u\in C([0,T],X\otimes (FS)')$ to \eqref{NLJ}. 
\end{theorem}

\begin{proof} According to Lemma \ref{Lema} for every $\alpha>\mathbf{0}$ the evolution equation \eqref{sistem 3} has a unique classical solution $u_\alpha\in C^1([0,T],X).$  Thus, the formal sum  $u(t,\omega)=\sum_{\alpha\in\mathcal{I}}u_\alpha(t)H_\alpha(\omega),$ $t\in[0,T],$ $\omega\in \Omega$ has coefficients that are all unique classical solutions to the corresponding deterministic equation \eqref{sistem 3}.  We have to prove that this formal sum is convergent in the sense of Definition \ref{def gen pr}. Note here that this solution will be unique  due to the uniqueness of the coordinatewise (classical) solutions $u_\alpha$, $\alpha\in \mathcal{I}$ given by  \eqref{sol hom} and \eqref{sol nehom} and due to the uniqueness in the Wiener--It\^o chaos expansion.  Additionally, once we show that this formal sum $u$ is convergent, its continuity  follows immediately since all the coefficients $u_\alpha$, $\alpha\in \mathcal I$ are continuous, see \eqref{C^k proc}.

Now we prove that $u\in C([0,T],X\otimes (FS)')$.  Let $u^0\in X\otimes (FS)'$ be an initial condition satisfying assumption (A2) which states that there exist $\tilde r\geq 2,$ $\tilde{p}\geq 0$ and $\tilde{K}>0$ such that $\sum_{\alpha\in\mathcal{I}}\|u^0_\alpha\|_X^2\alpha! (\tilde r^{|\alpha|^3}!)^{-1} (2\mathbb{N})^{-\tilde{p}\alpha}=\tilde{K}.$ Then there also exist $r\geq 2,$ $p\geq 0$ and $K>0$ such that $\sum_{\alpha\in \mathcal{I}}\|u^0_\alpha\|_X^2 \alpha!^2(r^{|\alpha|^3}!)^{-2} (2\mathbb{N})^{-2p\alpha}=K^2$, or equivalently
\begin{equation}\label{ocena p. uslova}
(\exists r\geq 2)\; (\exists p\geq 0)\;(\exists K>0)\; (\forall \alpha\in\mathcal{I})\quad \|u_\alpha^0\|_X\leq K \alpha! r^{|\alpha|^3}! (2\mathbb{N})^{p\alpha}.
\end{equation}
The inhomogeneous part $f\in C^1([0,T],X)\otimes (FS)'$ satisfies assumption (A3) which states that there exist $\tilde r\geq 2$ and $\tilde{p}\geq 0$ such that $\sum_{\alpha\in\mathcal{I}}\sup_{t\in[0,T]}\|f_\alpha(t)\|_X^2 \alpha!(\tilde r^{|\alpha|^3}!)^{-1} (2\mathbb{N})^{-\tilde{p}\alpha}<\infty.$ Then there exist $r\geq 2,$ $p\geq 0$ and $K>0$ such that
\begin{equation}\label{ocena ih. uslova}
\sup_{t\in[0,T]}\|f_\alpha(t)\|_X\leq K \alpha! r^{|\alpha|^3}! (2\mathbb{N})^{p\alpha},\quad \alpha\in \mathcal{I}.
\end{equation}
The coefficients $u_\alpha,\;\alpha\in\mathcal{I},$ $\alpha>\mathbf{0}$ of the solution $u$ are given by \eqref{sol hom} and \eqref{sol nehom}. Denote by
\[
L_\alpha:=\sup_{t\in[0,T]}\|u_\alpha(t)\|_X,\quad \alpha\in\mathcal{I}.
\]

The proof will be divided into three steps. In the first step we will derive a recurrent formula for $L_\alpha,$ $\alpha \in \mathcal{I}.$ Then, in the second step, an upper bound for $L_\alpha,$ $;\alpha \in \mathcal{I}$ will be provided and finally, the third step will be devoted to proving that  $u\in C([0,T],X)\otimes (FS)'.$

\textbf{Step 1.} For $\alpha=\mathbf{0}$ using \eqref{ocena u^0} one obtains 
\[
L_\mathbf{0}=\sup_{t\in[0,T]}\|u_\mathbf{0}(t)\|_X=M_\Phi,
\]
since the solution to \eqref{nelinearna det} satisfies assumption (A4).

Let $|\alpha|=1.$ Then $\alpha=\varepsilon_k,\;k\in \mathbb{N}$ and using \eqref{sol hom} we have that
\[
\|u_{\varepsilon_k} (t)\|_X\leq \|S_{\varepsilon_k,\Phi} (t,0)\|\| u_{\varepsilon_k}^0\|_X+\int_0^t\|S_{\varepsilon_k,\Phi} (t,s)\|\|f_{\varepsilon_k}(s)\|_X ds, \quad t\in [0,T].
\]
From \eqref{ocena S_alpha} we obtain that
\begin{equation}\label{ocen int}
\int_0^t\|S_{\alpha,\Phi}(t,s)\|ds\leq\int_0^t me^{w_\Phi(t-s)}ds =m\frac{e^{w_\Phi t}-1}{w_\Phi}\leq \frac{m}{w_\Phi}e^{w_\Phi T},\quad t\in[0,T],\quad\alpha>\mathbf{0}
\end{equation} 
and now \eqref{ocena S_alpha}, \eqref{ocena p. uslova} and \eqref{ocena ih. uslova} imply that
\begin{align}\label{ocena L_epsilon}
L_{\varepsilon_k}&=\sup_{t\in[0,T]}\|u_{\varepsilon_k} (t)\|_X\leq\sup_{t\in[0,T]}\Big\{\|S_{\varepsilon_k,\Phi} (t,0)\|\| u_{\varepsilon_k}^0\|_X+\sup_{s\in[0,t]}\|f_{\varepsilon_k}(s)\|_X\int_0^t\|S_{\alpha,\Phi}(t,s)\|ds\Big\}\nonumber\\
&\leq me^{w_\Phi T}K\varepsilon_k!r^{|\varepsilon_k|^3}!(2\mathbb{N})^{p\varepsilon_k} +\frac{m}{w_\Phi}e^{w_\Phi T}K \varepsilon_k!r^{|\varepsilon_k|^3}! (2\mathbb{N})^{p\varepsilon_k}=\widetilde m K \varepsilon_k!r^{|\varepsilon_k|^3}! (2\mathbb{N})^{p\varepsilon_k}, 
\end{align}
for all $k\in\mathbb{N},$ where $\widetilde m=(m+\frac{m}{w_\Phi})e^{w_\Phi T}.$

Now let $\alpha\in\mathcal{I}$ be such that $|\alpha|\geq 2.$
Then from \eqref{sol nehom} one obtains
\[
u_\alpha(t)=S_{\alpha,\Phi}(t,0)u^0_\alpha+\int_0^t S_{\alpha,\Phi}(t,s)\big[ r_{\alpha,\Phi}(s)+f_\alpha(s)\big]ds,\quad t\in [0,T].
\]
From this we have
\begin{align*}
L_\alpha &=\sup_{t\in[0,T]}\|u_\alpha(t)\|_X \\
& \leq \sup_{t\in[0,T]}\Bigg\{\|S_{\alpha,\Phi}(t,0)\|\|u^0_\alpha\|_X+\int_0^t \|S_{\alpha,\Phi}(t,s)\|\|r_{\alpha,\Phi}(s)\|_X ds \\
&\hspace*{7cm} +\int_0^t\|S_{\alpha,\Phi}(t,s)\|\|f_\alpha(s)\|_Xds\Bigg\}\\
&\leq \sup_{t\in[0,T]}\Bigg\{me^{w_\Phi t}\|u^0_\alpha\|_X+\sup_{s\in[0,t]}\|r_{\alpha,\Phi}(s)\|_X\int_0^t \|S_{\alpha,\Phi}(t,s)\|ds  \\
& \hspace*{7cm}+ \sup_{s\in[0,t]}\|f_\alpha (s)\|_X\int_0^t\|S_{\alpha,\Phi}(t,s)\|ds\Bigg\}.
\end{align*}
Using \eqref{ocen int} we obtain
\begin{align*}
L_\alpha &=\sup_{t\in[0,T]}\|u_\alpha(t)\|_X  \\
& \leq me^{w_\Phi T}\|u^0_\alpha\|_X+\frac{m}{w_\Phi}e^{w_\Phi T} \sup_{t\in[0,T]}\|r_{\alpha,\Phi}(t)\|_X  + \frac{m}{w_\Phi}e^{w_\Phi T}\sup_{t\in[0,T]}\|f_\alpha(t)\|_X\\
& \leq \widetilde m K\alpha!r^{|\alpha|^3}! (2\mathbb{N})^{p\alpha}+\frac{m}{w_ \Phi}e^{w_\Phi T} \sup_{t\in[0,T]}\|r_{\alpha,\Phi}(t)\|_X,
\end{align*} 
where again $\widetilde m=(m+\frac{m}{w_\Phi})e^{w_\Phi T}.$ Since $\widetilde m \geq \frac{m}{w_\Phi} e^{w_\Phi T},$  one easily obtains
\[
L_\alpha\leq \widetilde m \Big(K \alpha!r^{|\alpha|^3}! (2\mathbb{N})^{p\alpha}+ \sup_{t\in[0,T]}\|r_{\alpha,\Phi}(t)\|_X \Big).
\]
In the following estimate we will use that
\begin{align*}
\sup_{t\in [0,T]}\|\Phi^{(k)}\big(u_\mathbf{0}(t)\big)\|_X & =\sup_{t\in [0,T]} \left\| \sum_{n=0}^\infty a_{n+k} u^n_\mathbf{0}(t) \right\|_X \leq \sup_{t\in [0,T]} \left( \sum_{n=0}^\infty \|a_{n+k}\|_X\|u_\mathbf{0}(t)\|^n_X \right)\\
&\leq \sum_{n=0}^\infty \|a_{n+k}\|_X \left(\sup_{t\in [0,T]} \|u_\mathbf{0}(t)\|_X \right)^n \leq \sum_{n=0}^\infty \|a_{n+k}\|_X M^n_\Phi  \leq \varphi^{(k)}(M_\Phi) ,
\end{align*}
see  Section \ref{0.3}.  Now, using \eqref{ostatak} one obtains
\begin{align*}
& \sup_{t\in[0,T]}\|r_{\alpha,\Phi}(t)\|_X\leq \\
& \leq \sum_{k=2}^{|\alpha|}\frac{\varphi^{(k)}(M_\Phi)}{k!}\sum_{\mathbf{0}<\gamma_1<\alpha}\sum_{\mathbf{0}<\gamma_2<\gamma_1}\dots \sum_{\mathbf{0}<\gamma_{k-1}<\gamma_{k-2}} L_{\alpha-\gamma_1}L_{\gamma_1-\gamma_2}\dots L_{\gamma_{k-2}-\gamma_{k-1}}L_{\gamma_{k-1}}\\
& \leq \sum_{k=2}^{|\alpha|}\frac{\varphi^{(k)}(M_\Phi)}{k!}\sum_{\mathbf{0}<\gamma_1<\alpha}\sum_{\mathbf{0}<\gamma_2<\gamma_1}\dots \sum_{\mathbf{0}<\gamma_{k-1}<\gamma_{k-2}} L_{\alpha-\gamma_1}L_{\gamma_1-\gamma_2}\dots L_{\gamma_{k-2}-\gamma_{k-1}}L_{\gamma_{k-1}}.
\end{align*}

Finally for $|\alpha|\geq 2,\;\alpha\in\mathcal{I}$
\begin{align}\label{ocena L_alpha}
L_\alpha\leq \widetilde m &  \Big( K \alpha! r^{|\alpha|^3}! (2\mathbb{N})^{p\alpha}\nonumber\\
& +  \sum_{k=2}^{|\alpha|}\frac{\varphi^{(k)}(M_\Phi)}{k!}\sum_{\mathbf{0}<\gamma_1<\alpha}\sum_{\mathbf{0}<\gamma_2<\gamma_1}\dots \sum_{\mathbf{0}<\gamma_{k-1}<\gamma_{k-2}} L_{\alpha-\gamma_1}L_{\gamma_1-\gamma_2}\dots L_{\gamma_{k-2}-\gamma_{k-1}}L_{\gamma_{k-1}} \Big).
\end{align}
In the sequel, for the sake of clarity, we will provide an explicit form of the formula \eqref{ocena L_alpha} for the cases $|\alpha|=2$ and  $|\alpha|=3.$ For $\alpha\in \mathcal{I}$ such that $|\alpha|=2$ the following holds:
\[
L_\alpha\leq \widetilde m \Big(K\alpha! r^{|\alpha|^3}! (2\mathbb{N})^{p\alpha}+ \frac{\varphi''(M_\Phi)}{2!}\sum_{\mathbf{0}<\gamma_1<\alpha} L_{\alpha-\gamma_{1}}L_{\gamma_{1}} \Big), 
\]
which is for $\alpha=2\varepsilon_k,$ $k\in \mathbb{N}$
\[
L_{2\varepsilon_k}\leq \widetilde m \Big(K \;2!\; r^{|2\varepsilon_k|^3}!(2\mathbb{N})^{p 2 \varepsilon_k}+ \frac{\varphi''(M_\Phi)}{2!}L_{\varepsilon_k}^2 \Big)
\]
and for $\alpha=\varepsilon_k+\varepsilon_l,$ $k,l\in \mathbb{N},$ $k\neq l$
\[
L_{\varepsilon_k+\varepsilon_l}\leq \widetilde m \Big(K \;1!1!\; r^{|\varepsilon_k+\varepsilon_l|^3}! (2\mathbb{N})^{p(\varepsilon_k+\varepsilon_l)}+ \frac{\varphi''(M_\Phi)}{2!}2 L_{\varepsilon_k}L_{\varepsilon_l} \Big).
\]
Further, for $|\alpha|=3$ one obtains
\begin{align*}
L_\alpha\leq \widetilde m & \Big(K\alpha!r^{|\alpha|^3}!(2\mathbb{N})^{p\alpha} \\
& + \frac{\varphi''(M_\Phi)}{2!}\sum_{\mathbf{0}<\gamma_1<\alpha} L_{\alpha-\gamma_{1}}L_{\gamma_{1}} + \frac{\varphi^{(3)}(M_\Phi)}{3!}\sum_{\mathbf{0}<\gamma_1<\alpha}\sum_{\mathbf{0}<\gamma_2<\gamma_1} L_{\alpha-\gamma_1}L_{\gamma_1-\gamma_2}L_{\gamma_2}\Big), 
\end{align*}
which is for $\alpha=3\varepsilon_k,$ $k\in \mathbb{N}$
\[
L_{3\varepsilon_k}\leq \widetilde m \Big(K \;3!\; r^{3^3}!(2\mathbb{N})^{p 3 \varepsilon_k}+ \frac{\varphi''(M_\Phi)}{2!}2L_{2\varepsilon_k}L_{\varepsilon_k} + \frac{\varphi^{(3)}(M_\Phi)}{3!}L_{\varepsilon_k}^3 \Big),
\]
for $\alpha=2\varepsilon_k+\varepsilon_l,$ $k,l\in \mathbb{N},$ $k\neq l $
\begin{align*}
L_{2\varepsilon_k+\varepsilon_l}\leq \widetilde m & \Big(K \;2!1!  \; r^{3^3}!(2\mathbb{N})^{p(2\varepsilon_k+\varepsilon_l)}\\
& + \frac{\varphi''(M_\Phi)}{2!}\left( 2 L_{2\varepsilon_k}L_{\varepsilon_l} +2 L_{\varepsilon_k+\varepsilon_l}L_{\varepsilon_k}\right)+ \frac{\varphi^{(3)}(M_\Phi)}{3!} 3 L_{\varepsilon_k}^2 L_{\varepsilon_l}\Big)
\end{align*}
and for $\alpha=\varepsilon_k+\varepsilon_l+\varepsilon_s,$ $k,l,s\in \mathbb{N},$ $k\neq l, $ $k \neq s,$ $l\neq s$
\begin{align*}
L_{\varepsilon_k+\varepsilon_l+\varepsilon_s}  \leq &\widetilde m \Big(K\;1!1!1!  \;r^{3^3}!(2\mathbb{N})^{p(\varepsilon_k+\varepsilon_l+\varepsilon_s)} \\ \nonumber
& + \frac{\varphi''(M_\Phi)}{2!} \left( 2 L_{\varepsilon_k+\varepsilon_l}L_{\varepsilon_s} +2 L_{\varepsilon_k+\varepsilon_s}L_{\varepsilon_l} + 2 L_{\varepsilon_l+\varepsilon_s}L_{\varepsilon_k} \right)+ \frac{\varphi^{(3)}(M_\Phi)}{3!} 6 L_{\varepsilon_k}L_{\varepsilon_l}L_{\varepsilon_s} \Big).
\end{align*}

\textbf{Step 2.} This part of the proof is devoted to showing that there exist $r_0\geq 2$ and $p_0\geq 0$ such that
\begin{equation}\label{ih sp}
L_\alpha \leq  \alpha! (2\mathbb N)^{p_0\alpha}r_0^{|\alpha|^3}!\quad \mbox{for all}\quad \alpha\in \mathcal{I},\;|\alpha|\geq 1,
\end{equation}
where $p_0>p$ and $r_0>2(r+\widetilde{m}K+\widetilde{m}\widetilde{C}),$ $\widetilde C=(2C)^2$ for $C$ given in \eqref{const C}. The proof is derived by mathematical induction with respect to the length of the multiindex $\alpha\in \mathcal{I},$ i.e. with respect to $n=|\alpha|,$ $n\in \mathbb{N}.$

First, for $|\alpha|=1$ from \eqref{ocena L_epsilon} one directly obtains
\begin{equation}\label{ih 1sp}
L_{\alpha} \leq \alpha!\widetilde m K r^{|\alpha|^3}! (2\mathbb{N})^{p\alpha}\leq \alpha! (2\mathbb{N})^{p_0\alpha} r_0^{|\alpha|^3}!=\alpha!(2\mathbb{N})^{p_0\alpha} r_0!,
\end{equation}
since $p_0>p$ and $r_0>r+\widetilde m K.$

Next, for $|\alpha|=2,$ let $J_2$ be given as follows
\[
J_2=\sum_{\mathbf{0}<\gamma_1<\alpha} L_{\alpha-\gamma_{1}}L_{\gamma_{1}}.
\]
The sum over $\gamma_1\in \mathcal{I},$ where $\mathbf{0}<\gamma_1<\alpha,$ has $N(\alpha,2)\leq 2^{2|\alpha|}=2^4$ terms, as given in Lemma \ref{N(alfa, k)2}. Using this and the estimate \eqref{ih 1sp} for both terms $L_{\alpha-\gamma_{1}}$ and $L_{\gamma_{1}},$ since $|\alpha-\gamma_1|=|\gamma_1|=1,$ one obtains
\begin{align*}
J_2 & \leq  N(\alpha,2) (\alpha-\gamma_1)!(2\mathbb{N})^{p_0 (\alpha-\gamma_1)}r_0^{|\alpha-\gamma_1|^3}!(\gamma_1)!(2\mathbb{N})^{p_0 \gamma_1}r_0^{|\gamma_1|^3}!\leq  2^4 \alpha! (2\mathbb{N})^{p_0\alpha }r_0!r_0! \\
& \leq   2^4 \alpha!(2\mathbb{N})^{p_0\alpha }r_0^2!.
\end{align*}
Recall, the estimate \eqref{const C} says that since $\varphi $ is analytic and $\{u_0(t), t\in[0,T]\}$ is a compact set, then for all $k\in \mathbb N$ 
\[
\sup_{t\in[0,T]}\frac{\varphi^{(k)}(u_0(t))}{k!}\leq C^k, \quad \mbox{for some}\; C>0. 
\]
Hence, for $|\alpha|=2$ the following holds
\begingroup
\allowdisplaybreaks
\begin{align*}
L_{\alpha} \leq & \widetilde m \left( K\alpha!r^{|\alpha|^3}! (2\mathbb{N})^{p \alpha} + C^2 J_2\right)\leq \widetilde m \left( K\alpha!r^{|\alpha|^3}! (2\mathbb{N})^{p \alpha} + C^2 2^4\alpha! (2\mathbb{N})^{p_0\alpha}r_0^2! \right) \\
& \leq \alpha!(2\mathbb{N})^{p_0\alpha}r_0^{|\alpha|^3}!\left(\frac{\widetilde m K r^{|\alpha|^3}! }{r_0^{|\alpha|^3}!}+\frac{C^2 2^4 r_0^2!}{r_0^{|\alpha|^3}!}\right) \leq \alpha!(2\mathbb{N})^{p_0\alpha}r_0^{8}!\left(\frac{\widetilde m K r^8!}{r_0^{8}!}+\frac{C^2 2^4 r_0^2!}{r_0^{8}!}\right),
\end{align*}
\endgroup
since $p_0>p.$ Finally, using that $r_0^8!=(r_0^2r_0^6)!\geq r_0^2!r_0^6!$ and $r_0>2(r+\widetilde{m}K+\widetilde{m}\widetilde{C}),$ where $\widetilde C=(2C)^2,$ one obtains
\begin{align*}
L_{\alpha} & \leq \alpha!(2\mathbb{N})^{p_0\alpha}r_0^{8}!\left(\frac{\widetilde m K }{r_0^{8}}+\frac{C^2 2^4}{r_0^{6}!}\right)\leq  \alpha!(2\mathbb{N})^{p_0\alpha}r_0^{8}!\left (\frac{\widetilde m K}{r_0}+\frac{C^2 2^4}{r_0^2}\right)\\
& \leq \alpha!(2\mathbb{N})^{p_0\alpha}r_0^{8}!\left (\frac{1}{2}+\frac{1}{2}\right)=\alpha!(2\mathbb{N})^{p_0\alpha}r_0^{8}!.
\end{align*}
   
Now, let us assume that the inductive hypothesis \eqref{ih sp} is true for all $\alpha\in\mathcal{I}\setminus \{\mathbf{0}\}$ such that $|\alpha|\leq n.$ It remains to prove that then it also holds for $|\alpha|=n+1.$
   
We start with (\ref{ocena L_alpha}) for $|\alpha|\geq 2,$ $\alpha\in\mathcal{I}$:
\begin{align*}
L_\alpha\leq \widetilde m & \Big(K \alpha! r^{|\alpha|^3}! (2\mathbb{N})^{p\alpha} \\
& +  \sum_{k=2}^{|\alpha|}\frac{\varphi^{(k)}(M_\Phi)}{k!}\sum_{\mathbf{0}<\gamma_1<\alpha}\sum_{\mathbf{0}<\gamma_2<\gamma_1}\dots \sum_{\mathbf{0}<\gamma_{k-1}<\gamma_{k-2}} L_{\alpha-\gamma_1}L_{\gamma_1-\gamma_2}\dots L_{\gamma_{k-2}-\gamma_{k-1}}L_{\gamma_{k-1}} \Big).
\end{align*}
Let
\[
J_k=\sum_{\mathbf{0}<\gamma_1<\alpha}\sum_{\mathbf{0}<\gamma_2<\gamma_1}\dots \sum_{\mathbf{0}<\gamma_{k-1}<\gamma_{k-2}} L_{\alpha-\gamma_1}L_{\gamma_1-\gamma_2}\dots L_{\gamma_{k-2}-\gamma_{k-1}}L_{\gamma_{k-1}},\qquad k=2,\dots,|\alpha|. 
\]
This sum has $N(\alpha,k)\leq 2^{k|\alpha|}$ terms. Note that $1\leq |\alpha-\gamma_1|,|\gamma_1-\gamma_2|,\dots,|\gamma_{k-2}-\gamma_{k-1}|,|\gamma_{k-1}|\leq n.$ From \eqref{ih sp} and the fact that $ (\alpha-\gamma_1)!(\gamma_1-\gamma_2)!\cdots(\gamma_{k-2}-\gamma_{k-1})!(\gamma_{k-1})!\leq\alpha!$
this implies
\begin{align*}
J_k & \leq 2^{k|\alpha|} \alpha! (2\mathbb{N})^{p_0(\alpha-\gamma_1)}r_0^{|\alpha-\gamma_1|^3}!(2\mathbb{N})^{p_0(\gamma_1-\gamma_2)}r_0^{|\gamma_1-\gamma_2|^3}! \times \cdots\\
& \hspace*{4.5cm}\cdots \times (2\mathbb{N})^{p_0(\gamma_{k-2}-\gamma_{k-1})}r_0^{|\gamma_{k-2}-\gamma_{k-1}|^3}! (2\mathbb{N})^{p_0(\gamma_{k-1})}r_0^{|\gamma_{k-1}|^3}!\\
& \leq 2^{k|\alpha|} \alpha! (2\mathbb{N})^{p_0 \alpha}r_0^{|\alpha-\gamma_1|^3}!r_0^{|\gamma_1-\gamma_2|^3}!\cdots r_0^{|\gamma_{k-2}-\gamma_{k-1}|^3}! r_0^{|\gamma_{k-1}|^3}!.
\end{align*}   
Note that for $r_0\geq 2$ and $ |\alpha-\gamma_1|, |\gamma_1-\gamma_2|>0$ the  inequality \eqref{fakt nejed}  implies  
\[
r_0^{|\alpha-\gamma_1|^3}!r_0^{|\gamma_1-\gamma_2|^3}!\leq r_0^{|\alpha-\gamma_2|^3}!.
\]
Also for $\beta\in \mathcal{I}$ it holds
\[
(|\alpha|-|\beta|)^3+|\beta|^3\leq (|\alpha|-1)^3+1, \qquad \beta< \alpha, \; |\beta|\geq 1.
\]
The first inequality together with \eqref{fakt nejed}  implies
\[
r_0^{|\alpha-\gamma_1|^3}!r_0^{|\gamma_1-\gamma_2|^3}!\cdots r_0^{|\gamma_{k-2}-\gamma_{k-1}|^3}! r_0^{|\gamma_{k-1}|^3}! \leq r_0^{|\alpha-\gamma_{k-1}|^3+|\gamma_{k-1}|^3}!
\]
and now applying the second inequality, one obtains
\[
r_0^{|\alpha-\gamma_1|^3}!r_0^{|\gamma_1-\gamma_2|^3}!\cdots r_0^{|\gamma_{k-2}-\gamma_{k-1}|^3}! r_0^{|\gamma_{k-1}|^3}!\leq r_0^{(|\alpha|-1)^3+1}!.
\]
All these estimates imply
\[
J_k\leq  2^{k|\alpha|} \alpha!(2\mathbb{N})^{p_0 \alpha}r_0^{(|\alpha|-1)^3+1}!,\qquad k=2,\dots,|\alpha|. 
\]
Again using the estimate for the analytic function $\varphi$ one obtains 
\begin{align*}
L_\alpha & \leq \widetilde m \left(K\alpha! r^{|\alpha|^3}!(2\mathbb{N})^{p\alpha}+  \sum_{k=2}^{|\alpha|}C^k J_k \right)\\
& \leq \widetilde m \;\alpha!\left(Kr^{|\alpha|^3}!(2\mathbb{N})^{p\alpha}+  \sum_{k=2}^{|\alpha|}C^k 2^{k|\alpha|} (2\mathbb{N})^{p_0 \alpha}r_0^{(|\alpha|-1)^3+1}! \right)\\
& \leq \alpha!(2\mathbb{N})^{p_0 \alpha} \left(\widetilde m  Kr^{|\alpha|^3}! + \widetilde m r_0^{(|\alpha|-1)^3+1}! \sum_{k=2}^{|\alpha|}C^k 2^{k|\alpha|}  \right),
\end{align*}
provided by $p_0>p.$ Also, we need 
\[
\sum_{k=2}^{|\alpha|} C^k 2^{k|\alpha|}= \sum_{k=2}^{|\alpha|} (C2^{|\alpha|})^{k} \leq (C2^{|\alpha|})^{|\alpha|+1}=(2C)^{|\alpha|^2+|\alpha|}\leq \widetilde C^{|\alpha|^2},\qquad |\alpha|\geq 2,
\]
where $\widetilde C= (2C)^2.$ Continuing the series of estimates for $L_\alpha,\;|\alpha|=n+1$ one obtains
\begin{align*}
L_\alpha & \leq \alpha!(2\mathbb{N})^{p_0 \alpha} \left(\widetilde m  K r^{|\alpha|^3}!+ \widetilde m r_0^{(|\alpha|-1)^3+1}!  \widetilde C^{|\alpha|^2} \right)\\
&\leq \alpha!(2\mathbb{N})^{p_0 \alpha} r_0 ^{|\alpha|^3}!\left(\frac{\widetilde m  K r^{|\alpha|^3}!}{r_0^{|\alpha|^3}!}+ \frac{\widetilde m \widetilde C^{|\alpha|^2} r_0^{(|\alpha|-1)^3+1}!}{r_0^{|\alpha|^3}!} \right)\\
& \leq  \alpha!(2\mathbb{N})^{p_0 \alpha} r_0^{|\alpha|^3}!\left(\frac{\widetilde m  K}{r_0}+ \frac{\widetilde m \widetilde C^{(n+1)^2} r_0^{n^3+1}!}{r_0^{(n+1)^3}!} \right).
\end{align*}
Finally, using $a^{b+c}\geq a^b+a^c,$ for $a\geq 2$, we have
\[
\frac{r_0^{n^3+1}!}{r_0^{n^3+3n^2+3n+1}!}\leq \frac{1}{r_0^{3n^2+3n}!} \, .
\]
Thus, we obtain 
\begin{align*}
L_\alpha & \leq \alpha!(2\mathbb{N})^{p_0 \alpha} r_0^{|\alpha|^3}!\left(\frac{\widetilde m  K}{r_0}+ \frac{\widetilde m \widetilde C^{(n+1)^2} }{r_0^{3n^2+3n}!} \right)
\leq \alpha!(2\mathbb{N})^{p_0 \alpha} r_0^{|\alpha|^3}!\left(\frac{\widetilde m  K}{r_0}+ \frac{(\widetilde m \widetilde C)^{(n+1)^2} }{r_0^{(n+1)^2}} \right) \\
&\leq \alpha!(2\mathbb{N})^{p_0 \alpha} r_0^{|\alpha|^3}!,
\end{align*}
since $r_0>2(r+\widetilde{m}K+\widetilde{m}\widetilde{C}).$

\textbf{Step 3.} In this last step it is left to show that $u\in C([0,T],X)\otimes (FS)',$
i.e. that there exist $s\geq 2$  and $q\geq 0$ such that $u\in C([0,T],X)\otimes (FS)_{-s,-q}.$  Namely, it will be shown that for $u(t,\omega)=\sum_{\alpha\in\mathcal{I}} u_\alpha(t) H_\alpha(\omega)$ there exist $s\geq 2$ and $q\in \mathbb{N}$ such that
\begin{equation}\label{cilj sp}
\sum_{\alpha\in\mathcal{I}}\sup_{t\in [0,T]}\|u_\alpha(t)\|^2_X \alpha! (s^{|\alpha|^3}!)^{-1} (2\mathbb{N})^{-q\alpha}=\sum_{\alpha\in\mathcal{I}}L_\alpha^2  \alpha!(s^{|\alpha|^3}!)^{-1} (2\mathbb{N})^{-q\alpha}< \infty.
\end{equation}
Note that, for all $\alpha\in\mathcal{I},\;|\alpha|\geq 1$ by \eqref{ih sp} and \eqref{fakt nejed} it holds that
\[
L_\alpha^2  \leq \alpha!^2(2\mathbb{N})^{2 p_0 \alpha} (r_0^{|\alpha|^3}!)^2\leq \alpha!^2(2\mathbb{N})^{2 p_0 \alpha} r_0^{2|\alpha|^3}!.
\]
Moreover, since  $\alpha!\leq r_0^{|\alpha|^3}!$ then  for all  $\alpha\in\mathcal{I},$ $|\alpha|\geq 1$ it holds that
\[
L_\alpha^2  \leq (2\mathbb{N})^{2 p_0 \alpha} r_0^{4|\alpha|^3}!\,.
\]
Finally,  the series \eqref{cilj sp} is convergent since for $s=r_0^5$ and $q>2p_0+1$ the following holds
\begin{align*}
 \sum_{\alpha\in\mathcal{I}}L_\alpha^2 \alpha! (s^{|\alpha|^3}!)^{-1} (2\mathbb{N})^{-q\alpha} & \leq M_\Phi^2+\sum_{\alpha\in\mathcal{I}\setminus\{\mathbf{0}\}} (2\mathbb{N})^{2 p_0 \alpha} r_0^{4|\alpha|^3}! r_0^{|\alpha|^3}! (r_0^{5|\alpha|^3}!)^{-1} (2\mathbb{N})^{-q\alpha} \\
& <M_\Phi^2+\sum_{\alpha\in\mathcal{I}\setminus\{\mathbf{0}\}} (2\mathbb{N})^{(2 p_0-q) \alpha} <\infty, 
\end{align*}
where we used that $2p_0-q<-1$ and $M_\Phi$ is given in \eqref{ocena u^0}.  
\end{proof}

\begin{remark}
Note that if assumption (A4) is violated in so far that the governing equation \eqref{nelinearna det} has more than one solution (as many nonlinear PDEs have non-unique solutions), then each of its solutions $u_\mathbf{0}(t),$ $t\in[0,T]$ will generate a corresponding family of operators $B_{\alpha,\Phi}(t)=A_{\alpha}+\Phi'(u_\mathbf{0}(t))\,Id$, $t\in[0,T],$ $\alpha\in\mathcal I\setminus\{\mathbf{0}\}$, which further generates a corresponding semigroup $S_{\alpha,\Phi}(t,s),$ $0\leq s\leq t\leq T,$ $\alpha\in\mathcal I\setminus\{\mathbf{0}\}$ that provides the coordinatewise solutions $u_\alpha(t)$, $t\in[0,T],$ $\alpha\in\mathcal I\setminus\{\mathbf{0}\}$, via \eqref{sistem 3}. Hence each solution $u_\mathbf{0}(t),$ $t\in[0,T]$ of the PDE \eqref{nelinearna det} will provide its own family of recursively generated coefficients $u_\alpha(t)$, $t\in[0,T],$ $\alpha\in\mathcal I\setminus\{\mathbf{0}\}$, and therefore also a stochastic solution $u(t),$ $t\in[0,T]$ to the initial SPDE \eqref{NLJ}. The nonlinear stochastic equation \eqref{NLJ} will therefore have exactly as many solutions as the equation \eqref{nelinearna det} in  (A4).
\end{remark}

\section{Applications}\label{2.0}

In this section we present some examples of stochastic nonlinear equations that are of the form \eqref{PNLJ} and on which our proposed method and hence also Theorem \ref{main} can be applied. We fully work out the first example (the heat equation with an exponential nonlinearity) as an illustration and note that all subsequent  examples can be solved by a similar pattern following the formula derived in the proof of Theorem \ref{main}.

\subsection{Stochastic heat equations with exponential nonlinearity}
\label{2.2}

The stochastic reaction-diffusion equation with Wick-exponential type nonlinearity of the form 
\begin{align*}
u_t &= \triangle u + a+ be^{\lozenge u} + f, \quad a, b \not=0,\, ab<0,\\
u(0) &= u^0
\end{align*}
can be solved by applying Theorem \ref{main}.  This equation is also called the stochastic   Fujita--Gelfand equation in combustion. Here the nonlinearity $\Phi^{\lozenge}(u) = e^{\lozenge u}=\sum_{n=1}^\infty\frac1{n!}u^{\lozenge n}$ is a convex function that  appears in explosion dynamics.  The corresponding deterministic  problem is used to model convection in a heat transfer process when exothermic chemical reactions occur in the fluid body \cite{Jones}. 

We analyze the one-dimensional case when $f$ is an $\mathbb R-$valued process with zero expectation. For technical simplicity let $a=2, $ $b=-2$ and we consider an initial condition of the form $u^0(x,\omega)=u_\mathbf{0}^0(x)+N(\omega)=u_\mathbf{0}^0(x)+\sum_{\alpha\in\mathcal I}H_{\alpha}(\omega)$, where $N$ denotes a spatial random noise with expectation of the form $u_\mathbf{0}^0(x)=-2\ln(1+e^{-x})$ and constant noise part given by the coefficients $u_\alpha^0(x)=1$, $\alpha\in\mathcal I\setminus\{\mathbf{0}\}$. Equation \eqref{nelinearna det} hence reduces to
\[ 
\frac{d}{dt} u_\mathbf{0}(t,x) = \frac{d^2}{dx^2} u_\mathbf{0}(t,x)+2-2e^{u_\mathbf{0}(t,x)},\quad u_\mathbf{0}(0,x)=-2\ln(1+e^{-x}). 
\]
By \cite{Polyanin} this equation has a solution to the form
\begin{equation}\label{mojanulta} 
u_\mathbf{0}(t,x) = -2\ln(1+e^{-x-t})
\end{equation}
that is indeed  of $C^1([0,T],\mathbb R)$ class, hence (A4) is satisfied. For the nonlinearity $\Phi(u)=2-2e^{u}$ we have $\Phi'(u)=-2e^u$, hence by \eqref{sistem 2}, the corresponding family of operators is
\[
B_{\alpha,\Phi}(t)=\triangle+\Phi'(u_\mathbf{0}(t)) = \frac{d^2}{dx^2}-2(1+e^{-x-t})^{-2}, 
\] 
actually not depending on $\alpha$. The associated semigroup to this operator will be of the form $S_{\alpha,\Phi}(t,s)=e^{-2t(1+e^{-x-s})^{-2}}Z_t,$ where $Z_t$ is the semigroup associated to the Laplace operator. Hence, $S_{\alpha,\Phi}$  also does not depend on $\alpha$ and reads as
\[ 
S_{\alpha,\Phi}(t,s) \, g(x) = e^{-2t(1+e^{-x-s})^{-2}}\frac{1}{\sqrt{4\pi t}}\int_{\mathbb R}g(x-y)e^{-\frac{y^2}{4t}}dy,\quad\mbox{ for }\;  g\in L^p(\mathbb R).
\]
Its action onto the initial condition will be
\[ 
S_{\alpha,\Phi}(t,0)\, u_\alpha^0(x) = e^{-2t(1+e^{-x})^{-2}}\frac{1}{\sqrt{4\pi t}}\int_{\mathbb R}e^{-\frac{y^2}{4t}}dy = e^{-2t(1+e^{-x})^{-2}}. 
\]
By \eqref{ostatak} we have
\begin{align*}
r_{\alpha,\Phi}(t,x)= -2(1+e^{-x-t})^{-2}\sum_{k=2}^{|\alpha|}\frac{1}{k!}&\sum_{\mathbf{0}<\gamma_1<\alpha}  \sum_{\mathbf{0}<\gamma_2<\gamma_1} \dots \sum_{\mathbf{0}<\gamma_{k-1}<\gamma_{k-2}}\\
& u_{\alpha-\gamma_1}(t,x)u_{\gamma_1-\gamma_2}(t,x)\dots u_{\gamma_{k-2}-\gamma_{k-1}}(t,x) u_{\gamma_{k-1}}(t,x),
\end{align*} 
and $g_{\alpha,\Phi}(t,x)=r_{\alpha,\Phi}(t,x)+f_\alpha(t,x)$. Finally, by \eqref{sol nehom} we obtain the coefficients $u_\alpha$ of the solution recursively  as:
\begin{align}\label{mojaalfta}\nonumber
& u_{\alpha}(t,x)= e^{-2t(1+e^{-x})^{-2}}+\int_0^t e^{-2t(1+e^{-x-s})^{-2}}\Big(\frac{1}{\sqrt{4\pi t}}\int_{\mathbb R}
\Big[ f_\alpha(s,x-y)-2(1+e^{-x-y-s})^{-2} \times \\\nonumber
&\times\sum_{k=2}^{|\alpha|}\frac{1}{k!}\sum_{\mathbf{0}<\gamma_1<\alpha}  \sum_{\mathbf{0}<\gamma_2<\gamma_1} \dots \sum_{\mathbf{0}<\gamma_{k-1}<\gamma_{k-2}}
u_{\alpha-\gamma_1}(s,x-y)u_{\gamma_1-\gamma_2}(s,x-y)\dots u_{\gamma_{k-1}}(s,x-y)\Big]\\
&\times e^{-\frac{y^2}{4t}}dy\Big)ds,
\end{align} 
for $\alpha>  \mathbf{0}$. The final solution is thus fully determined by \eqref{mojanulta}  and \eqref{mojaalfta} and it is given by the chaos expansion $u(t,x,\omega)=\sum_{\alpha\in\mathcal I}u_\alpha(t,x)H_\alpha(\omega)$.

\subsection{Stochastic heat equations with polynomial  nonlinearities}
\label{2.1}

Stochastic nonlinear heat equations with Wick-polynomial nonlinearities of the form 
\begin{align*}
u_t &= \triangle u + a u + b u^{\lozenge n} + f, \quad a, b \not=0, \, n>1\\
u(0) &= u^0
\end{align*}
where $f\in C([0,T],X)\otimes (FS)'$ is a generalized stochastic process and $u^0\in X\otimes (FS)'$ is a generalized random variable, belong to the class of problems \eqref{PNLJ}.  These  problems arise in quantum mechanics, plasma physics and mathematical biology. 
 
Particularly, the Fisher--KPP equation is one of the most important models in mathematical biology and ecology. It describes a population that evolves under the combined effects of spatial diffusion and local logistic growth and saturation, for example in  migration and population behavior in biological invasion.  In one spatial dimension the stochastic Fisher--KPP equation is given by 
\begin{align*}
u_t &= \triangle u + u -  u^{\lozenge 2}\\
u(0) &= u^0.
\end{align*}

The stochastic Allen–-Cahn equation  is of the form
\begin{align*}
u_t &= \triangle u + u -  u^{\lozenge 3} + f, 
\\
u(0) &= u^0. 
\end{align*}
It is used to describe phase separation and the evolution of interfaces between the phases for systems without mass conservation taking into account the effects due to the thermal fluctuations of the system. 

Amplitude equations are used as models for  different stripes patterns such as ripples in sand, stripes of seashells arising in a variety of natural phenomena.  One of the most well known amplitude equation is the Newell--Whitehead--Segel equation  which describes the appearance of the stripe pattern in two dimensional systems.   The Newell--Whitehead--Segel  equation models the interaction of the effect of the diffusion term with the nonlinear effect of the reaction term.  Moreover, this type of equations describes the dynamical behavior near the bifurcation point of the Rayleigh--Benard convection of binary fluid mixtures. 

The stochastic Newell--Whitehead--Segel equation is an reaction-diffusion equation of the form 
\begin{align*}
u_t &= \triangle u + a u - b u^{\lozenge n} + f, \quad a, b > 0, \quad n\geq 2\\
u(0) &= u^0. 
\end{align*}

\subsection{Stochastic heat equations with logarithmic nonlinearities}
\label{2.3}

Stochastic heat equations with logarithmic nonlinearities  of the form 
\begin{align*}
u_t &= \triangle u +  \log ^{\lozenge}(1+ u) + f\\
u(0) &= u^0,
\end{align*}
with $\log ^{\lozenge}(1+ u) = \sum_{n=1}^\infty(-1)^{n+1}\frac{1}{n}\;u^{\lozenge n}$,
as well as of the form 
\begin{align*}
u_t &= \triangle u + u \lozenge \log^{\lozenge} |u| + f\\
u(0) &= u^0
\end{align*}
belong to the class \eqref{PNLJ} and can be solved by Theorem \ref{main}.

\subsection{Other types of nonlinearities}
\label{2.4}

Other types of nonlinearities $\Phi$ which are  analytic functions such as trigonometric functions or hyperbolic functions  can be considered in the form \eqref{PNLJ}.  For example, a nonlinear stochastic equation of the form
\begin{align*}
u_t &= \triangle u + \cos^{\lozenge}(u)+ \, \cosh^{\lozenge}( u )+ f, \quad a \not=0\\
u(0) &= u^0
\end{align*}
with 
\begin{align*}
\Phi^{\lozenge} (u) &= \cos^{\lozenge} u + \cosh^{\lozenge} u \\&= \sum_{n=0}^\infty (-1)^n\frac{1}{(2n)!} u^{\lozenge(2n)} + \sum_{n=0}^\infty \frac{1+(-1)^n}{2(n)!} u^{\lozenge n} = \sum_{n=0}^\infty \frac{1}{(4n)!}\;u^{\lozenge(4n)} 
\end{align*}
satisfies the conditions of Theorem \ref{main} and has a unique almost classical solution in  $C([0,T],X)\otimes (FS)'$.

\subsection{Stochastic reaction-diffusion equation with multiplicative noise}
\label{2.5}

In \cite{Flandoli} global existence and uniqueness is proved for a stochastic reaction-diffusion equation with polynomial nonlinearity \eqref{special Phi} and multiplicative noise in a bounded domain.  These equations can be generalized to problems  of the form 
\begin{align}\label{stoch RD}
u_t &= \triangle u +  \Phi^{\lozenge}(u) + u \lozenge W_t\\
u(0) &= u^0, \nonumber
\end{align}
where $W_t$ denotes the white noise process.  Our proposed method can be applied also to the problem \eqref{stoch RD}. Equations of the form \eqref{stoch RD} with nonlinearities  $\Phi$ which are not necessarily polynomial, but rather belong to another class of analytic functions  (exponential, logarithmic, trigonometric or hypergeometric)  can be solved with our method. Additionally, instead of a multiplicative white noise, we can consider a general type of noise terms $u\lozenge G_t$ following the same procedure as we did in \cite{Milica}.


\section*{Conclusion and discussions}
In this paper we have studied abstract stochastic Cauchy problems of the form  $u_t= {\textbf A} u  + \Phi^{\lozenge}(u) + f$, $u(0)=u^0$, driven by unbounded linear operators that generate $C_0$ semigroups; including but not limited to second order elliptic differential operators that arise in most real-world applications. We have established an appropriate theoretical framework and proven existence and uniqueness of a solution in appropriate weighted spaces of generalizes stochastic processes. The approach can be generalized also to stochastic equations of the form 
\begin{align*}
	u_t &= {\textbf A} u +   {\textbf B} \lozenge u + \Phi^{\lozenge}(u) + f\\
	u(0) &= u^0. 
\end{align*} 
Multiplicative noise nonlinearities can be also considered in the form of a more general equation
\begin{align*}
	u_t &= {\textbf A} u +   {\textbf B} \lozenge \Psi^{\lozenge}(u) + \Phi^{\lozenge}(u) + f\\
	u(0) &= u^0, 
\end{align*}
with two analytic functions $\Psi$ and $\Phi$, and two operators ${\textbf A}$ and ${\textbf B}$, separately modeling the
part with ordinary multiplication and the part with the Wick product. These proofs can be carried out using the guidelines provided in \cite{Milica} and \cite{Milica2} as well as the results of this paper.

Future work will focus on numerical simulations and the implementation of explicit solutions through calculated coefficients of the chaos expansion. Examples of practical applications of the results obtained include the stochastic Navier-Stokes equation and other diffusion models from an applied perspective, such as \cite{Stefan}. We emphasize again that the primary benefit of our approach is that, through this procedure, we obtain explicit solutions in terms of an easily truncable infinite series. The truncation can be made to approximate the theoretical solution reasonably well by applying appropriate error estimates on the truncation, which are based on the norm estimates of $(FS)'-$spaces established in the main theorem. This is crucial for numerical simulations and practical applications of different diffusion and reaction-diffusion models.

Further investigation could focus on a more thorough comparison between the Wick and ordinary products in situations where the operator ${\textbf A}$ has sufficiently regular coefficients to yield solutions in test spaces rather than generalized spaces. This would enable the consideration of models with the structure $u_t= {\textbf A} u + \Phi(u) + f$, $u(0)=u^0$. This is a pertinent question, and studies that look into the relationship between the two kinds of products typically only look at one kind of equation (e.g. the Burgers equation in \cite{Grothaus}, the fluid pressure equation in \cite{HOUZ} or fluid flow equation in \cite{Wan}), whereas this paper studies a whole class of nonlinear SPDEs, so a comparison of that kind is currently out of scope. We note that each specific equation would have to be handled separately and each specific equation would result in a different type of error estimate between the two types of products. 

Finally, our future work will also include a study of stochastic evolution equations  $u_t=\textbf{M} u+\Phi^\lozenge (u)+f,$ $u(0)=u^0,$ driven by a nonlinear operator $\textbf{M}$ which is nonlinear perturbation of an unbounded linear operator $\textbf{A}.$





\begin{thebibliography}{99999}

\bibitem{Albeverio}
Albeverio,  S.,  Haba,  Z., Russo, F.  (2001).
A two space dimensional semilinear heat equation perturbed by white noise.
\emph{Probab.  Theory Related Fields.} 121: 319--366.  
DOI: 10.1007/s004400100153
	 
\bibitem{Alfaro}
Alfaro,  M.,  Carles,  R.  (2017). 
Superexponential growth or decay in the heat equation with a logarithmic nonlinearity. 
\emph{Dyn.  Partial  Differ.  Equ.} 14(4): 343--358.  
DOI: 10.4310/DPDE.2017.v14.n4.a2

\bibitem{Potthoff}
Benth, F., Deck, T., Potthoff, J. (1997). A white noise approach to a class of non-linear
stochastic heat equations, \emph{Jour. Func. Anal.} 146: 382--415.
DOI: 10.1006/jfan.1996.3048
	  
\bibitem{Bertacco}
Bertacco,  F.  (2021).
Stochastic Allen-Cahn equation with logarithmic potential.
\emph{Nonlinear Anal. } 202: 1--22.  
DOI: 10.1016/j.na.2020.112122
	 
\bibitem{log}
Bialynicki-Birula,  I.,  Mycielski,  J.  (1976).
Nonlinear wave mechanics. 
\emph{Ann.  Physics.} 100(1/2): 62--93.  
DOI: 10.1016/0003-4916(76)90057-9
	 
\bibitem{blum} 
Blum,  E.  K.  (1955).
A theory of analytic functions in Banach algebras.
\emph{Trans.  Amer.  Math.  Soc.} 78: 343--370.  
DOI: 10.1090/S0002-9947-1955-0069405-2
	 
\bibitem{Allen-Cahn1}
Calatroni,  L.,  Colli,  P.  (2013).
Global solution to the Allen--Cahn equation with singular potentials and dynamic boundary conditions. 
\emph{Nonlinear Anal.} 79: 12--27.  
DOI: 10.1016/j.na.2012.11.010

\bibitem{DaPratoTubaro}
Da Prato,  G., Tubaro, L.  (2023).
Wick powers in stochastic PDEs: an introduction.
In: Hilbert, A., Mastrogiacomo, E., Mazzucchi, S., R\"udiger, B., Ugolini, S.  eds. \emph{Quantum and Stochastic Mathematical Physics. } Springer Proceedings in Mathematics $\&$ Statistics, vol 377. Springer, Cham.
DOI: 10.1007/978-3-031-14031-0$\_$1
	 	
\bibitem{Flandoli}
Flandoli,  F.  (1991).
A stochastic reaction-diffusion equation with multiplicative noise. 
\emph{Appl.  Math.  Lett.} 4(4): 45--48.  
DOI: 10.1016/0893-9659(91)90052-W
	
\bibitem{Allen-Cahn2}
Gal,  C.  G.,  Grasselli,  M.  (2008).
The non-isothermal Allen-Cahn equation with dynamic boundary conditions.  
\emph{Discrete Contin.  Dyn.  Syst.} 22(4): 1009--1040.  
DOI: 10.3934/dcds.2008.22.1009

\bibitem{Garban}
Garban,  C.  (2020).
Dynamical Liouville.  
\emph{J. Funct. Anal.} 278(6): 54 pp.
DOI: 10.1016/j.jfa.2019.108351

\bibitem{Wick}
Gjessing,  H.,  Holden,  H.,  Lindstr\o m,  T., \O ksendal,  B.,  Ub\o e J.,  Zhang, T. S.  (1993).
The Wick product.
In: Niemi,  H.,  H\"ognas,  G.,  Shiryaev, A. N.,  Melnikov, A. V.  eds.  \emph{Frontiers  in  Pure  and  Applied  Probability. Vol. 1.  Proceedings of the Third Finnish-Soviet Symposium on Probability Theory and Mathematical Statistics, Turku, Finland, August 13-16, 1991.} Moscow,  Russia: Science Publisher,  pp.  29--67.
DOI: 10.1515/9783112314203-005

\bibitem{Grothaus}
Grothaus,  M.,  Kondratiev,  Yu.  G.,  Streit, G.  L.  (2000).
Scaling limits for the solution of Wick type Burgers equation. 
\emph{Random Oper. Stoch. Equ.} 8(1): 1--26. 
DOI: 10.1515/rose.2000.8.1.1

\bibitem{Hairer}
Hairer, M.  (2014).
A theory of regularity structures. 
\emph{Invent. Math.} 198(2): 269--504. 
DOI: 10.1007/s00222-014-0505-4
	
\bibitem{Hariharan}
Hariharan,  G.,  Kannan,  K.  (2010).
Haar wavelet method for solving some nonlinear parabolic equations.
\emph{J.  Math.  Chem.} 48(4): 1044--1061.  
DOI: 10.1007/s10910-010-9724-0
	
\bibitem{Hida}
Hida,  T.,  Kuo,  H.  H.,  Potthoff,  J.,  Streit,  L.  (1993).
\emph{White noise. An infinite-dimensional calculus.}
Dordrecht: Kluwer Academic Publishers Group.

\bibitem{HO1}
Holden,  H.,  Lindstr\o m,  T.,  \O ksendal,  B.,  Ub\o e,  J.,  Zhang, T.  (1995).
The stochastic Wick-type Burgers equation.  
In: Etheridge,  A.  ed.  \emph{Stochastic Partial Differential Equations. }
Cambridge,  MA: Cambridge University Press,  pp. 141--161.
DOI: 10.1017/CBO9780511526213.010

\bibitem{HO2}
Holden,  H.,  Lindstr\o m,  T.,  \O ksendal,  B.,  Ub\o e,  J.,  Zhang, T.  (1995).
The pressure equation for fluid flow in a stochastic medium. 
\emph{Potential Anal.} 4(6): 655--674.
DOI: 10.1007/BF02345830

\bibitem{HOUZ}
Holden,  H.,  \O ksendal,  B.,  Ub\o e,  J.,  Zhang,  T.  (2010).
\emph{Stochastic partial differential equations. A modeling, white noise functional approach. Second edition. } New York,  NY: Springer. 

\bibitem{HO3}
Hu Y.,  \O ksendal, B.  (1996).
Wick approximation of quasilinear stochastic differential equations.  In: K{\"o}rezlio\v glu, H.,  \O ksendal, B.,  \"Ust\"unel,  A.S.  eds.  \emph{Stochastic Analysis and Related Topics,} Vol.  5. Boston, MA: Birkh\"auser,  pp. 203--231.
DOI: 10.1007/978-1-4612-2450-1\_10

\bibitem{Jentzen}
Jentzen,  A.,  Kloeden,  P.,  Winkel,  G.  (2011).
Efficient simulation of nonlinear parabolic SPDEs with additive noise. 
\emph{Ann.  Appl.  Probab.} 21(3): 908--950.  
DOI: 10.1214/10-AAP711

\bibitem{Jones}
Jones,  D.  R.  (1973).
The dynamic stability of confined, exothermically reacting fluids.  
\emph{Int.  J.  Heat Mass Transfer.} 16(1): 157--167.  
DOI: 10.1016/0017-9310(73)90259-7

\bibitem{kritika}
Kaligotla,  S., Lototsky,  S. V.  (2011).
Wick product in the stochastic Burgers equation: A curse or a cure?
\emph{Asymptot. Anal. } 75(3-4): 145--168.
DOI: 10.3233/ASY-2011-1058

\bibitem{23}
Kato, T.  (1973).
Linear evolution equations of "hyperbolic" type.  II.
\emph{J.  Math.  Soc.  Japan.} 25: 648--666.  
DOI: 10.2969/jmsj/02540648

\bibitem{Milica} 
Levajkovi\'c, T.,  Pilipovi\' c,  S.,  Sele\v si,  D.,  \v Zigi\' c,  M.  (2015).
Stochastic evolution equations with multiplicative noise.
\emph{Electron.  J.  Probab.} 20: Article 19.  
DOI: 10.1214/EJP.v20-3696

\bibitem{Milica2} 
Levajkovi\'c, T.,  Pilipovi\' c,  S.,  Sele\v si,  D.,  \v Zigi\' c,  M.  (2018).
Stochastic evolution equations with Wick-polynomial nonlinearities.
\emph{Electron.  J.  Probab.} 23: Article 116.  
DOI: 10.1214/18-EJP241

\bibitem{LPS}
Levajkovi\'c, T.,  Pilipovi\' c,  S.,  Sele\v si,  D. (2011).
The Stochastic Dirichlet Problem Driven by the Ornstein–Uhlenbeck Operator: Approach by the Fredholm Alternative for Chaos Expansions. 
\emph{Stoch. Anal. Appl.} 29: 317--331. 
DOI: 10.1080/07362994.2011.548998

\bibitem{LiuQiao}
Liu,  Z.,  Qiao,  Z.  (2020).
Strong approximation of monotone stochastic partial differential equations driven by white noise.  
\emph{IMA J.  Numer.  Anal.} 40(2): 1074--1093.  
DOI: 10.1093/imanum/dry088

\bibitem{Sergey}
Lototsky, S., Rozovskii, B., Wan, X. (2010). Elliptic equations of higher stochastic order.
\emph{ESAIM: Modélisation mathématique et analyse numérique}. 44: 1135--1153. 
DOI: 10.1051/m2an/2010055

\bibitem{rough}
Lyons,  T.  J.  (1998).
Differential equations driven by rough signals. 
\emph{Rev. Mat. Iberoam.} 14(2): 215--310. 
DOI: 10.4171/RMI/240

\bibitem{Mikulevicius}
Mikulevicius,  R.,  Rozovskii,  B.  (2012).
On unbiased stochastic Navier-Stokes equations.  
\emph{Probab.  Theory Related Fields.} 154: 787--834.  
DOI: 10.1007/s00440-011-0384-1

\bibitem{NZ} 
Neidhardt,  H.,  Zagrebnov,  V.  A.  (2009).
Linear non-autonomous Cauchy problems and evolution semigroups. 
\emph{Adv.  Differential Equations.} 14(3/4): 289--340.  
URL: https://projecteuclid.org:443/euclid.ade/1355867268

\bibitem{Ober}
Oberguggenberger,  M.,  Russo,  F.  (1998).
Nonlinear SPDEs: Colombeau solutions and pathwise limits.
In: Decreusefond, L., \O ksendal,  B.,  Gjerde, J.,  {\"U}st{\"u}nel,  A. S.,  eds.  \emph{Stochastic Analysis and Related Topics VI. } Progress in Probability,  Vol. 42.  Boston,  MA: Birkh{\"a}user, pp.  319--332.
DOI: 10.1007/978-1-4612-2022-0\_14

\bibitem{Habib}
Ouerdiane, H. (2002).
Algebras of infinite dimensional holomorphic functions and application to probability. \emph{Trends in infinite-dimensional analysis and quantum probability. Proceedings of a symposium held at the Research Institute for Mathematical Sciences, Kyoto}, 158--171.

\bibitem{Pazy}
Pazy,  A.  (1983).
\emph{Semigroups of linear operators and applications to partial differential equations.} 
New York, NY: Springer.  

\bibitem{GRPW}
Pilipovi\'c, S.,  Sele\v si,  D.  (2007).
Expansion theorems for generalized random processes, Wick products and applications to stochastic differential equations. 
\emph{Infin.  Dimens.  Anal.  Quantum Probab.  Relat.  Top.} 10(1): 79--110.  
DOI: 10.1142/S0219025707002634

\bibitem{ps}
Pilipovi\'c, S.,  Sele\v si,  D.  (2010).
On the generalized stochastic Dirichlet problem. I. The stochastic weak maximum principle.
\emph{Potential Anal.} 32(4): 363--387.  
DOI: 10.1007/s11118-009-9155-3

\bibitem{Polyanin}
Polyanin,  A.  D.,  Zaitsev,  V.  F.  (2012).
\emph{Handbook of nonlinear partial differential equations,} 2nd ed. 
Boca Raton,  FL: CRC Press.

\bibitem{Dora}
Sele\v{s}i,  D.  (2008).
Algebra of Generalized Stochastic Processes and the Stochastic Dirichlet problem. 
\emph{Stoch.  Anal.  Appl. } 26(5): 978--999.
DOI: 10.1080/07362990802286053

\bibitem{Stefan}
Sele\v si, D., To\v si\'c, S. (2024). A fractional stochastic model for aerosol transmission of fluid droplets and virus exposure in closed space, Preprint.

 \bibitem{Theting}
Theting, T.  G.  (2004). 
Numerical Solution of Wick-Stochastic Partial Differential Equations.  
In: Albeverio, S.,  Boutet de Monvel, A., Ouerdiane, H.,   eds.
\emph{Proceedings of the International Conference on Stochastic Analysis and Applications. } Dordrecht, The Netherlands: Springer,  pp.  303--349. 
DOI:10.1007/978-1-4020-2468-9\_18

\bibitem{Mikulevicius2}
Venturi,  D., Wan,  X. , Mikulevicius,  R.,  Rozovskii,  B. L.,  Karniadakis,  G. E.  (2013).
Wick-Malliavin approximation to nonlinear stochastic partial differential equations: analysis and simulations. 
\emph{Proc.  R.  Soc. A} 469(2158): 20130001.
DOI: 10.1098/rspa.2013.0001

\bibitem{51}
Yosida,  K.  (1965/66).
Time dependent evolution equations in a locally convex space.
\emph{Math.  Ann.} 162: 83--86.  
DOI: 10.1007/BF01361935

\bibitem{Wan}
Wan, X., Rozovskii, B. (2013). The Wick-Malliavin approximation of elliptic problems with log-normal random coefficients.
\emph{SIAM J. Sci. Comput.} 35(5):  pp. A2370--A2392.
DOI: 10.1137/130918605

\end{thebibliography}
\end{document}